\documentclass[a4paper,11pt]{article}

\usepackage[margin=3cm,footskip=1cm]{geometry}
\usepackage{amsmath,amssymb,amsthm,mathtools,bm,dsfont,mathrsfs,amsfonts}
\usepackage{microtype}

\usepackage{graphicx}
\usepackage{tikz}
\usepackage{booktabs}
\usepackage{xcolor}

\usepackage{enumitem}
\usepackage{url}
\usepackage[above,below,section]{placeins}

\usepackage{marginnote}

\makeatletter
\newcommand{\margin}[1]{%
  \ifmmode
    \marginnote{#1}%
  \else
    \marginpar{#1}%
  \fi
}
\makeatother

\usepackage[colorlinks=true,linkcolor=blue,citecolor=blue,urlcolor=blue]{hyperref}
\usepackage[numbers,sort&compress]{natbib}
\usepackage{cleveref}

\numberwithin{equation}{section}

\theoremstyle{plain}
\newtheorem{theorem}{Theorem}[section]

\theoremstyle{definition}
\newtheorem{definition}[theorem]{Definition}

\newtheorem{remark}[theorem]{Remark}
\newtheorem{example}[theorem]{Example}

\newcommand{\e}{\mathrm{e}}
\newcommand{\dd}{\mathrm{d}}
\newcommand{\PP}{\mathbb{P}}
\newcommand{\EE}{\mathbb{E}}

\title{Matrix Representations for Scale Functions of Spectrally Negative Lévy Processes with Rational Jumps}
\author{Osvaldo Angtuncio Hern\'andez, Oscar Peralta}
\date{}

\begin{document}
\maketitle

\begin{abstract}
For a spectrally negative L\'evy process with Laplace transform $\psi$, the $q$-scale function is characterized as the function whose Laplace transform is $(\psi(\cdot)-q)^{-1}$. 
It has applications in fluctuation theory, for example, exit problems and  first hitting probabilities. 
It is also used in areas like ruin theory, risk theory, continuous state branching processes and optimal control. 
In this paper, we extend the scale function representation of \cite{Ivanovs2021} from spectrally negative Lévy processes with phase-type jumps to the general case of matrix-exponential jumps. The extension is non-trivial because the probabilistic arguments employed by Ivanovs rely on an embedding to a Markov-modulated Brownian motion, a framework that does not accommodate the algebraic generality of matrix-exponential distributions. 
We overcome this limitation by embedding the L\'evy process into a stochastic fluid process modulated by a rational arrival process (RAP), a class of continuous-valued Markov processes driven by orbit processes.
This approach yields iterative schemes related to those of \cite{Ivanovs2021} to provide a simple and explicit formula for the scale function. Our method gives the same fixed point when restricted to the phase-type case, and demonstrates the utility of orbit representations in analytical problems beyond the phase-type setting.
\end{abstract}


\tableofcontents
\newpage

\section{Introduction}

The scale function of a spectrally negative Lévy process (snLp) is a fundamental object in fluctuation theory, providing explicit expressions for two-sided exit problems, first passage probabilities, overshoot distributions, occupation measures, and numerous other quantities of interest in ruin theory, queueing, risk theory, continuous state branching processes, fragmentation processes, L\'evy processes conditioned to stay positive, optimal control and optimal stopping (see  \citep{KuznetsovKyprianouRivero2013} and the references therein; see also \cite{AvramPalmowskiPistorius2007,MR1988433,MR2023021,MR1175272,MR1053243,MR410961,MR2164035}). 
While scale functions exist for general spectrally negative processes, tractable representations emerge when the jump distribution has a rational Laplace transform \cite{lewis2008wiener,KuznetsovKyprianouRivero2013}. This rationality condition is equivalent to the negative jump distribution being matrix-exponential, see \cite[Theorem 4.1.10]{BladtNielsen2017}.

Matrix-exponential distributions \citep{Cox1955} form a dense subset of distributions on the positive half-line and are characterized by densities of the form $f(x) = \bm{\alpha} \e^{\bm{T}x} \bm{t}$ for $x \ge 0$, where $\bm{\alpha}$ is a row vector, $\bm{T}$ is a square matrix, and $\bm{t}$ is a column vector. 
Representations $(\bm{\alpha}, \bm{T}, \bm{t})$ may have complex entries, but can always be transformed to real representations satisfying $\bm{\alpha}\bm{1} = 1$ and $\bm{t} = -\bm{T}\bm{1}$, where $\bm{1}$ denotes a column vector of ones of appropriate size \citep{asmussen1996renewal}.
The entries of these objects can be positive or negative. 
Nevertheless, the assumption on the  representation $(\bm{\alpha}, \bm{T}, \bm{t})$ is that $f$ is a well-defined density.
These distributions strictly generalize the class of phase-type distributions \citep{Neuts1981}, which arise when $\bm{T}$ is a subgenerator (that is, off-diagonal entries nonnegative, row sums nonpositive). In this case the representation admits a Markovian interpretation (see Section \ref{sectionMEAndPHDns} for such a description). We refer to \cite{BladtNielsen2017} for a comprehensive modern treatment of matrix-exponential and phase-type distributions.

The work of \cite{Ivanovs2021} establishes an elegant representation of the scale function for spectrally negative Lévy processes with phase-type jumps. Prior to this work, scale functions for phase-type jump processes were computed via root-finding algorithms \citep{asmussen2004russian,pistorius2006maxima,egami2014phase,asmussen2014levy}. Ivanovs shows that the scale function can be expressed in terms of a transition rate matrix $G$ solving a quadratic matrix equation, with the solution computable via a fast monotone iterative scheme that offers significant computational advantages, particularly when the number of phases is large.

Our objective is to extend these results to the case of general spectrally negative Lévy processes with matrix-exponential jumps. This extension is far from routine. The probabilistic framework employed by Ivanovs fundamentally relies on embedding the Lévy process into a Markov-modulated Brownian motion \citep{ivanovs2010markov,d2012two}, a construction that requires the jump distribution to admit a phase-type representation with an underlying continuous-time Markov chain. In general, matrix-exponential distributions that are not phase-type lack such a Markovian structure (although some examples can be found in \cite{peralta2023markov}), and the probabilistic arguments used by Ivanovs cannot be directly applied. This difficulty is closely related to the broader problem of understanding when rational or matrix-exponential objects admit a genuinely Markovian representation; see \cite{asmussen2022ph,telek2022two,peralta2026rational} for recent discussion.

The key insight that enables our extension comes from recent developments in the theory of stochastic fluid processes modulated by a rational arrival process (RAP) introduced by \cite{BeanNguyenNielsenPeralta2022}. These processes substitute the Markov-modulated Brownian motion of \citep{Ivanovs2021}, allowing us to obtain a broader result, and at the same time using a simpler embedding. RAP-modulated stochastic fluid processes are driven by orbit processes, a class of piecewise deterministic Markov processes that extend continuous-time Markov chains to accommodate matrix-exponential interarrival times and other algebraic constructs   \cite{AsmussenBladt1999,bean2010quasi,BuchholzTelek2013}, without requiring finite state-space Markovian modulation. 
In simple words, the RAP-modulated stochastic fluid process is composed of a pair $(R,\bm{A}):=\{R_t,\bm{A}_t\}_{t\geq 0}$, which are a level process and an orbit process, respectively. The evolution of the level process at each time $t$, is governed by the value of the  modulator $\bm{A}_t$. 
Orbit processes provide the natural state space for rational arrival processes, which generalize Markovian arrival processes  in the same way that matrix-exponential distributions generalize phase-type distributions.
Markovian arrival processes are defined using a Markov process $\{J_t\}$ with $p<\infty$ states, acting as a modulator, and such that at time intervals where $J_t=i$, the arrivals are Poisson at rate $\beta_i$. In addition, there is a certain probability that an arrival occurs at a jump from $i$ to $j\neq i$.

An underutilized aspect of the classical stochastic fluid process framework is the systematic correspondence between one-sided exit problems for Markov-modulated Brownian motion and those for stochastic fluid processes, established through the Wiener-Hopf factorization of Brownian states \cite{NguyenPeralta2019}. Adapting the technique of the latter allows results derived in the Brownian setting with rational jumps to be translated directly to the RAP-modulated fluid process context without the need to establish a separate study. The hitting probabilities we require can be expressed through Riccati matrix equations studied by \cite{BeanNguyenNielsenPeralta2022}, which yield iterative schemes related to those of Ivanovs, converging to the same fixed point in the phase-type case.

Our contribution can be understood as a matrix-exponential extension of Ivanovs' result. We demonstrate that the transition from phase-type to matrix-exponential distributions, while losing several convenient probabilistic interpretations, preserves the essential algebraic structure needed to compute scale functions. The orbit framework provides the appropriate setting for this generalization, extending the toolkit of analytical methods available for matrix-exponential models.

Another important contribution of our paper stems from the embedding approach. While Ivanovs embeds the L\'evy process into a Markov-modulated Brownian motion (MMBM), our approach avoids working with diffusive paths altogether. Instead, we embed the process into a RAP-modulated stochastic fluid process, whose paths are piecewise linear, and apply the Wiener-Hopf factorization of Brownian paths over exponentially-distributed horizons to reduce the Brownian component to an equivalent non-diffusive problem. This reduction is exact, not a limiting approximation, and is precisely analogous to how \cite[Section 4]{NguyenPeralta2019} study MMBM via stochastic fluid processes. As a result, the only information required from the trajectories of the L\'evy process, are the values just before and after each jump, together with the infimum on each inter-jump interval.

A key advantage of matrix-exponential distributions over phase-type distributions concerns the minimality of representations. For a distribution with a rational Laplace transform of degree $p$, a matrix-exponential representation of order $p$ is always achievable and is minimal by definition, whereas minimal phase-type representations remain an open problem and in general require a number of phases strictly larger than $p$ \cite{o1990characterization}. Consequently, working with matrix-exponential representations yields numerically more attractive algorithms, as the representation dimension is as small as theoretically possible. Beyond parsimony, recent interest in concentrated matrix-exponential distributions, which serve as lower-dimensional alternatives to Erlang distributions in approximating the Dirac delta measure, has renewed attention to the class of matrix-exponential distributions that are not phase-type \cite{horvath2016concentrated,horvath2020high,horvath2020numerical,meszaros2022concentrated}. 

We emphasize that phase-type and matrix-exponential distributions share the same analytic tractability, admitting closed-form expressions for densities, distribution functions, and Laplace transforms, which facilitate explicit numerical evaluation. As our main result demonstrates, computations such as differentiation and integration of the scale function can be performed analytically, thereby reducing numerical error in computation.

The remainder of the paper proceeds as follows. 
Section~\ref{sec:matrix_exponential_distributions} defines the main objects of interest, the L\'evy process, its scale function, and the matrix-exponential distribution. 
The main results are also stated here.
Section~\ref{sectionRationalToME} establishes the constructive procedure for obtaining representations of matrix-exponential distributions with rational Laplace transforms. Section~\ref{sec:orbit_processes_and_rap_modulated_stochastic_fluid_processes} introduces orbit processes and RAP-modulated stochastic fluid processes, establishing the framework for our embedding construction. Section~\ref{sec:main_result_scale_function_representation} contains the proofs of the main results, deriving the scale function representation of the spectrally negative L\'evy process, for both bounded and unbounded variation cases.
The iterative algorithm to obtain the parameters necessary for the scale function representation is given in Section~\ref{sec:main_result_scale_function_representation}.
Finally, Section~\ref{sec:conclusion} concludes with remarks on the significance of the extension and directions for future work.

\section{Preliminaries and Main Results}\label{sec:matrix_exponential_distributions}
Here we review some of the basics on the two main objects of this manuscript: Lévy processes and matrix-exponential distributions. Most of this information can be found in \cite{Kyprianou2014,BladtNielsen2017}.
\subsection{Spectrally Negative Lévy Processes with Rational Jumps}
Consider a spectrally negative Lévy process $X$, that is, a continuous-time process starting at zero with independent and stationary increments, having paths continuous to the right and with left limits, and with only negative jumps. By the Lévy-Khintchine representation, its Laplace exponent $\psi(\theta): = \log \EE[\e^{\theta X_1}]$, can be written as
\begin{equation}\label{eq:levy_khintchine}
\psi(\theta) = c\theta + \frac{1}{2}\sigma^2\theta^2 + \int_{0}^\infty\big(\e^{-\theta x} - 1+\theta x{\bf 1}(x\leq 1)\big)\Pi(\dd x),\qquad \mbox{$\theta\geq 0$,}
\end{equation}
where $c \in \mathbb{R}$, $\sigma \ge 0$, and $\Pi$ is a Lévy measure on $(0,\infty)$. When $\Pi$ has finite positive mass, say $\Pi(0,\infty) = \lambda$, the process has compound Poisson jumps with rate $\lambda$  and jump size distribution $F(\dd x) = \Pi(\dd x)/\lambda$. The Laplace exponent then takes the form
\begin{equation}\label{eq:psi_compound_poisson}
\psi(\theta) = d\theta + \frac{1}{2}\sigma^2\theta^2 + \lambda\int_{0}^\infty\big(\e^{-\theta x}-1\big) F(\dd x),\qquad \mbox{ $\theta\geq 0$},
\end{equation}where $d\in \mathbb{R}$ is the drift in the representation
\begin{equation}\label{eqnXTrajectories}
X_t=dt+\sigma B_t-\sum_{j=1}^{N_t}C_j,
\end{equation}with $\{C_j\}_j$  i.i.d. having distribution $F$, and $N$ is a Poisson process with rate $\lambda$.
We say that $X$ has \emph{rational jumps}  if the jump size distribution $F$ has finite positive mass and a rational Laplace-Stieltjes transform:
\begin{equation}\label{eq:rational_LST}
L(\theta) := \int_0^\infty \e^{-\theta x} F(\dd x) = \frac{b_1 + b_2\theta + \cdots + b_p\theta^{p-1}}{\theta^p + a_1\theta^{p-1} + \cdots + a_{p-1}\theta + a_p},\qquad \mbox{$\theta\geq 0$,}
\end{equation}
where $a_1, \ldots, a_p, b_1, \ldots, b_p \in \mathbb{R}$, the numerator and denominator have no common factors, and $\deg(\text{numerator}) < \deg(\text{denominator}) = p\in \mathbb{N}$. 
In this case we call $p$ the order of the rational Laplace-Stieltjes transform.
The assumption of the $a_i$'s and $b_i$'s is such that the RHS of \eqref{eq:rational_LST} is well-defined and nonnegative for all $\theta\geq 0$. 
Since $L(\theta)$ is a Laplace-Stieltjes transform, we have $L(0) = 1$, which implies $b_1 = a_p \neq 0$. The Laplace exponent of $X$ then becomes
\begin{equation}\label{eq:psi_rational}
\psi(\theta) = d\theta + \frac{\sigma^2}{2}\theta^2 + \lambda(L(\theta) - 1) = d\theta + \frac{\sigma^2}{2}\theta^2 + \lambda\left(\frac{b_1 + b_2\theta + \cdots + b_p\theta^{p-1}}{\theta^p + a_1\theta^{p-1} + \cdots + a_p} - 1\right).
\end{equation}
We are interested in the case where the $C_j$'s have a matrix-exponential distribution, which is described in the next section.
But first, we describe the scale function, which is the core function to analyze in this paper. 

For any $q\geq 0$, 
the $q$-scale function $W^{(q)} : \mathbb{R} \to [0,\infty)$ is characterized by its Laplace transform:
\begin{equation}\label{eq:scale_transform}
\int_0^\infty \e^{-\theta x} W^{(q)}(x)\,\dd x = \frac{1}{\psi(\theta) - q}, \qquad \theta > \Phi_q,
\end{equation}
where $\Phi_q \ge 0$ is the largest root of $\psi(\theta) = q$.
For $x<0$ it holds that $W^{(q)}(x)=0$.

\subsection{Matrix-Exponential and Phase-Type Distributions}\label{sectionMEAndPHDns}
Distributions with rational Laplace-Stieltjes transforms admit a matrix representation that facilitates both analytical and computational work. We now introduce this representation.
\begin{definition}
A distribution on $(0,\infty)$ is called \emph{matrix-exponential} (ME) if its density can be written as
\begin{equation}\label{eq:ME_density}
f(x) = \bm{\alpha} \e^{\bm{T}x} \bm{t}, \qquad x > 0,
\end{equation}
where $\bm{\alpha}$ is a $1 \times p$ row vector, $\bm{T}$ is an $p \times p$ matrix, $\bm{t}$ is an $p \times 1$ column vector, and $p\in \mathbb{N}$. 
The triplet $(\bm{\alpha}, \bm{T}, \bm{t})$ is called a \emph{matrix-exponential representation} of order $p$.
\end{definition}
The Laplace transform of a matrix-exponential distribution admits the rational form
\begin{equation}\label{eq:ME_Laplace}
\int_0^\infty \e^{-\theta x} f(x)\,\dd x = \bm{\alpha}(\theta \bm{I} - \bm{T})^{-1}\bm{t} = \frac{P(\theta)}{Q(\theta)}, \qquad \theta\geq 0,
\end{equation}
where $P$ and $Q$ are polynomials with $\deg P < \deg Q \le p$. 
Note that the last equality in \eqref{eq:ME_Laplace} holds true since $(\theta \bm{I} - \bm{T})^{-1}=\operatorname{adj}(\theta \bm{I} - \bm{T})\det(\theta \bm{I} - \bm{T})^{-1}$, the denominator is a polynomial of degree $p$, and $\operatorname{adj}(\theta \bm{I} - \bm{T})$ is composed of all square minors of size $p-1$ obtained from $\theta \bm{I} - \bm{T}$, and thus polynomials of degree at most $p-1$. 
The parameter $p$ in a matrix-exponential representation need not be minimal; a representation is called \emph{minimal} if $\deg Q = p$.
Matrix-exponential distributions strictly generalize phase-type distributions. A distribution is phase-type (PH) if and only if it admits a matrix-exponential representation $(\bm{\alpha}, \bm{T}, \bm{t})$ where $\bm{T}$ is a subgenerator, $\bm{\alpha}$ is a probability vector ($\bm{\alpha} \ge \bm{0}$, $\bm{\alpha}\bm{1} = 1$), and $\bm{t}=-\bm{T}\bm{1}$. In this case, the representation has a Markovian interpretation: $\bm{T}$ is the transient subgenerator of a continuous-time Markov chain on $p$ transient states, $\bm{\alpha}$ is the initial distribution, and $\bm{t}$ specifies the absorption rates to a unique absorbing state.
We call such a continuous-time Markov chain the modulator Markov process of the PH distribution, or simply the modulator. 
Note that, since for matrix-exponential distributions the entries of the representation $(\bm \alpha,\bm T,\bm t)$ can be negative, such Markovian interpretation in general is still an open problem (see \cite{peralta2023markov}). 
Concentrated matrix-exponential distributions are particularly useful as they can efficiently approximate deterministic delays and other low-variance phenomena using significantly fewer dimensions than standard phase-type representations (such as Erlang distributions) \cite{horvath2016concentrated,horvath2020high,horvath2020numerical,meszaros2022concentrated}.

\subsubsection{Main results}

Following \cite{Ivanovs2021}, we employ the identity
\begin{equation}\label{eq:Ivanovs_identity}
W^{(q)}(x) = \frac{1}{\psi'(\Phi_q)} \Big( \e^{\Phi_q x} - \PP(\tau_{\{-x\}} < e_q) \Big), \qquad x \ge 0,
\end{equation}
where $e_q$ is an independent exponential time with rate $q$ (with the convention $e_0=\infty$ almost surely), and        
$\tau_{\{-x\}}=\inf\{t>0:X_t=-x\}$ is the first hitting time of level $-x$ (see \cite{MR2126965,MR2946445,KuznetsovKyprianouRivero2013}). From the above, our task reduces to computing $\PP(\tau_{\{-x\}} < e_q)$. 

\begin{theorem}\label{teoScaleFnBV}
Let $X$ be a L\'evy process as in \eqref{eqnXTrajectories}, with $\sigma=0$ and $d>0$. Then, for any $x\geq 0$ and $q\geq 0$ such that $\psi'(\Phi_q)>0$ we have
\begin{equation}\label{eq:scale_bv}
W^{(q)}(x) = \frac{1}{\psi'(\Phi_q)} \Big( \e^{\Phi_q x} - \bm{\Psi}\, e^{\bm{G}x}\,\bm{\nu} \Big), \qquad x \ge 0,
\end{equation}where $\bm{G}: = \bm{T} + \bm{t}\,\bm{\Psi}$, $\bm{\nu}:=(\Phi_q \bm I - \bm{T})^{-1}\bm{t}$ with $\bm I$ the identity matrix.
The matrix $\bm{\Psi}$ is a solution to 
\begin{equation}\label{eq:Psi_bv}
\bm{\Psi} \big( (\lambda + q)\bm{I} - d\bm{T} - d\,\bm{t}\,\bm{\Psi} \big) = \lambda\, \bm{\alpha}.
\end{equation}Finally, $\bm{\Psi}$ can be obtained as the limit of the recursive equation
\begin{equation}\label{eq:iteration_bv1}
\bm{\Psi}_n
:=
\Big(\frac{\lambda}{d}\bm{\alpha} + (\bm{\Psi}_{n-1}\bm{t})\bm{\Psi}_{n-1}\Big)
\left(\frac{\lambda + q}{d}\bm I - \bm{T}\right)^{-1},
\end{equation}with  $\bm{\Psi}_0=\bm{0}$.
\end{theorem}

Observe that the iteration in \eqref{eq:iteration_bv1} is a variant of the scheme proposed by Ivanovs \cite[Section 4.2]{Ivanovs2021} for computing the limiting vector $\bm{\Psi} = \lim_{n \to \infty} \bm{\Psi}_n$ (see Remark \ref{rem:bv_iteration_diff} below). 
Indeed, the primary difference between the two schemes is that Ivanovs's approach evaluates $(\bm{\Psi}_{n-1}\bm{t})\bm{\Psi}_{n}$ rather than $(\bm{\Psi}_{n-1}\bm{t})\bm{\Psi}_{n-1}$. In the context of classical stochastic fluid processes, both iterations converge to the same fixed point. Consequently, the limiting equation serves as the matrix-exponential analogue of Ivanovs's Riccati equation.

Now we state the matrix-exponential extension of \cite[Theorem 1]{Ivanovs2021} in the unbounded variation case.

While the explicit construction of the scale function for the unbounded variation case relies on the detailed description of RAP-modulated processes developed in the subsequent sections, we present the main result here for immediate reference. Note that the algorithm provides a fully explicit computation of all required parameters.

\begin{theorem}\label{teoScaleFnUV}
Let $X$ be a L\'evy process as in \eqref{eqnXTrajectories}, with $\sigma>0$ and $d\in \mathbb{R}$. 
Consider $\bm{G}: = \bm{C}^- + \bm{D}^{-+}\bm{\Psi}$ which, under the block structure given below in \eqref{eq:params_ubv}-\eqref{eq:params_ubv1}, takes the explicit form $\bm{G} = \begin{pmatrix} -a & \bm{b} \\ \bm{t} & \bm{T} \end{pmatrix}$. Define $\bm{\nu}:=(\Phi_q \bm{I}-\bm{T})^{-1}\bm{t}$ and $\bm{V}:=\binom{1}{\bm{\nu}}$. Let $\bm{e}_1^-:=(1,0,\dots,0)^\intercal\in\mathbb{R}^{p+1}$ be the first canonical basis vector.
Then, for any $x\geq 0$ and $q\geq 0$ such that $\psi'(\Phi_q)>0$ we have
\begin{equation}\label{eq:scale_ubv}
W^{(q)}(x) = \frac{1}{\psi'(\Phi_q)}\Big(e^{\Phi_q x} - (\bm{e}_1^-)^\intercal\,e^{\bm{G}x}\,\bm{V}\Big), \qquad x \ge 0.
\end{equation}
The matrix $\bm{\Psi}$ can be 
recovered from the pair $(a,\bm{b})$ via
\begin{equation}\label{eq:Psi_from_ab}
\bm{\Psi} = \frac{1}{\omega_{\lambda+q}}\begin{pmatrix}\omega_{\lambda+q} - a & \quad \bm{b}\end{pmatrix},
\end{equation}
where $\omega_{\lambda+q} := \big(\sqrt{d^2 + 2\sigma^2 (\lambda+q)} + d\big)\sigma^{-2}$, $\eta_{\lambda+q} := \big(\sqrt{d^2 + 2\sigma^2 (\lambda+q)} - d\big)\sigma^{-2}$, and $(a,\bm b)$ are a solution to the system of equations
\begin{equation}\label{eq:ivanovs_scalar}
\sigma^2a^2-2da-2(\lambda+q)+\sigma^2\bm b\bm t=0,
\end{equation}
\begin{equation}\label{eq:ivanovs_vector}
\sigma^2\bm{b}\left(\big(a-2d\sigma^{-2}\big)\bm I-\bm{T}\right) =2\lambda\bm{\alpha}.
\end{equation}
Finally, such a pair can be obtained as the limit of the recursive equations
\begin{equation}\label{eq:b_recursion}
\bm{b}_n = \Big( \frac{2\lambda}{\sigma^2}\,\bm{\alpha} + (\omega_{\lambda+q}-a_{n-1})\bm{b}_{n-1} \Big) \big(\eta_{\lambda+q}\bm{I} - \bm{T}\big)^{-1}, \qquad n \ge 1,
\end{equation}
\begin{equation}\label{eq:a_recursion}
a_n = \omega_{\lambda+q} - \frac{\sigma^2}{2\sqrt{d^2+2\sigma^2(\lambda+q)}}\Big((\omega_{\lambda+q} - a_{n-1})^2 + \bm{b}_n\bm{t}\Big), \qquad n \ge 1,
\end{equation}
with $a_0=\omega_{\lambda+q}$, $\bm{b}_0=\bm{0}$.
\end{theorem}

The limiting system \eqref{eq:ivanovs_scalar}--\eqref{eq:ivanovs_vector} coincides with Equations (7) of \cite{Ivanovs2021}, but the iterative algorithm \eqref{eq:b_recursion}--\eqref{eq:a_recursion} differs from the alternating scheme of \cite[Section 4.2]{Ivanovs2021}: ours is a linear iteration arising naturally from the Sylvester recursion \eqref{eq:psi_sylvester} given below of the RAP-modulated fluid framework, whereas Ivanovs's is a nonlinear fixed-point iteration on the limiting equations. Notably, the Sylvester recursion requires only a single matrix inversion of $(\eta_{\lambda+q}\bm{I} - \bm{T})$, completely avoiding the repeated sequential inversions demanded by Ivanovs's variant. When the jump distribution is phase-type, both algorithms converge to the same fixed point $(a,\bm{b})$. But only our algorithm is currently available in the RAP-modulated setting, as the richer toolkit of classical stochastic fluid models that enables Ivanovs's variant has not yet been developed for RAP-modulated processes.


\section{From Rational Laplace-Stieltjes Transforms to Matrix-Exponential Representations}\label{sectionRationalToME}

In this section we show that
every distribution with a rational Laplace-Stieltjes transform  \eqref{eq:rational_LST} is matrix-exponential. 
The latter is done via a
constructive procedure for obtaining a matrix-exponential representation from a given rational form.

\subsection{The Companion Representation}

Given the rational Laplace-Stieltjes transform \eqref{eq:rational_LST} of order $p$, we construct a companion matrix representation as follows. Define the square companion matrix
\begin{equation}\label{eq:companion_S}
\bm{S} = \begin{pmatrix}
0 & 1 & 0 & \cdots & 0 & 0 \\
0 & 0 & 1 & \cdots & 0 & 0 \\
0 & 0 & 0 & \ddots & 0 & 0 \\
\vdots & \vdots & \vdots & \ddots & \ddots & \vdots \\
0 & 0 & 0 & \cdots & 0 & 1 \\
-a_p & -a_{p-1} & -a_{p-2} &  \cdots & -a_2 & -a_1
\end{pmatrix}.
\end{equation}
Note that its characteristic polynomial is $\det(\theta \bm{I} - \bm{S}) = \theta^p + a_1\theta^{p-1} + \cdots + a_{p-1}\theta + a_p$, which follows directly by induction. 
Indeed, expanding the determinant on the first column we have $\det(\theta \bm{I} - \bm{S})=\theta \det(\theta \bm{I} - \bm{S}_{p-1})+a_p$, where $\bm{S}_{p-1}$ is the matrix $\bm{S}$ with the first row and column erased. 
Define the initial vector $\bm{\beta} = (b_1, \ldots, b_p)$ and the \emph{closing vector} of size $p\times 1$
\begin{equation}\label{eq:s_vector}
\bm{s} = \begin{pmatrix}
0 \\
0 \\
\vdots \\
0 \\
1
\end{pmatrix}.
\end{equation}
Then, defining
\begin{equation}\label{eq:companion_density}
f_R(x): = \bm{\beta} \e^{\bm{S}x} \bm{s}, \qquad x > 0,
\end{equation}one can show that its Laplace-Stieltjes transform $L_R$ satisfies \begin{equation}\label{eq:companion_LST}
L_R(\theta)=\bm{\beta}(\theta \bm{I} - \bm{S})^{-1}\bm{s} = \frac{b_1 + b_2\theta + \cdots + b_p\theta^{p-1}}{\theta^p + a_1\theta^{p-1} + \cdots + a_p},\qquad \mbox{$\theta\geq 0$}.
\end{equation}
Since a distribution on $[0,\infty)$ is uniquely characterized by its Laplace-Stieltjes transform, the required representation holds.

Note that to compute the numerator of \eqref{eq:companion_LST}, since $\bm{s}$ is the $p$-th standard basis vector, $\bm{\beta}\operatorname{adj}(\theta \bm{I} - \bm{S})\bm{s}$ extracts the dot product of $\bm{\beta}$ with the last column of the adjugate matrix. By expanding the minors, one has $(\operatorname{adj}(\theta \bm{I} - \bm{S}))_{j,p} = (-1)^{j+p}\det((\theta \bm{I} - \bm{S})_{(p,j)})$, where $(\theta \bm{I} - \bm{S})_{(p,j)}$ means the matrix $\theta \bm{I}-\bm S$ with row $p$ and column $j$ deleted. The submatrix is block triangular, and its determinant evaluates exactly to $(-1)^{p-j}\theta^{j-1}$. Hence $\bm{\beta}\operatorname{adj}(\theta \bm{I} - \bm{S})\bm{s}=\sum_{j=1}^p \beta_j\theta^{j-1}$, matching the required numerator.

If the order of the Laplace-Stieltjes transform is minimal, then the companion representation has the advantage of being unique.
However, it does not necessarily satisfy the normalization  $\bm{s} = -\bm{S}\bm{1}$ as in the phase-type setting, which is convenient for applications in fluctuation theory, as used by \cite{Ivanovs2021}.
In the next section, we describe one representation that yields such a normalization.

\subsection{The Standardized ME Representation}

We now describe a (non-unique) procedure for transforming the companion representation into a standardized form characterized by $\bm{t} = -\bm{T}\bm{1}$ and $\bm{\alpha}\bm{1} = 1$. These conditions ensure that the closing vector $\bm{t}$ is determined by $\bm{T}$ alone and that $\bm{\alpha}$ integrates to one, without imposing any non-negativity constraints on the entries of $(\bm{\alpha}, \bm{T})$. The construction uses a similarity transformation with a specific matrix $\bm{M}$ \citep{asmussen1996renewal}.

Recalling that $a_p \neq 0$, then the companion matrix $\bm{S}$ is invertible. Direct calculation yields
\begin{equation}\label{eq:S_inverse}
\bm{S}^{-1} = \begin{pmatrix}
-\frac{a_{p-1}}{a_p} & -\frac{a_{p-2}}{a_p} & -\frac{a_{p-3}}{a_p} & \cdots & -\frac{a_1}{a_p} & -\frac{1}{a_p} \\
1 & 0 & 0 & \cdots & 0 & 0 \\
0 & 1 & 0 & \cdots & 0 & 0 \\
\vdots & \vdots & \ddots & \ddots & \vdots & \vdots \\
0 & 0 & 0 & \ddots & 0 & 0\\
0 & 0 & 0 & \cdots & 1 & 0
\end{pmatrix},
\end{equation}
and therefore
\begin{equation}\label{eq:S_inverse_s}
\bm{S}^{-1}\bm{s} = \begin{pmatrix}
-\frac{1}{a_p} \\
0 \\
0 \\
\vdots \\
0
\end{pmatrix}.
\end{equation}
Given the companion representation $(\bm{\beta}, \bm{S}, \bm{s})$ as in the previous section, define the transformation matrix

\begin{equation}\label{eq:M_matrix}
\bm{M} \;=\;
\begin{pmatrix}
\frac{1}{a_p} & 0 & 0 & \cdots & 0 \\
-1 & 1 & 0 & \cdots & 0 \\
0 & -1 & 1 & \cdots & 0 \\
\vdots & \vdots & \vdots & \ddots & \vdots \\
0 & 0 & 0 & -1 & 1
\end{pmatrix}.
\end{equation}

The matrix $\bm{M}$ is constructed so that $\bm{M}\bm{1}=-\bm{S}^{-1}\bm{s}$.
Using the latter, we define the transformed representation
\begin{equation}\label{eq:phlike_rep}
\bm{\alpha}: = \bm{\beta} \bm{M}, \qquad \bm{T}: = \bm{M}^{-1}\bm{S}\bm{M}, \qquad \bm{t}: = \bm{M}^{-1}\bm{s}.
\end{equation}
This ensures that
\[
-\bm{T}\bm{1}=\bm{M}^{-1}\bm{s}=\bm{t},
\]
and moreover,
\[
\bm{\alpha}\bm{1}=\bm{\beta}\bm{M}\bm{1}=-\bm{\beta}\bm{S}^{-1}\bm{s}=\frac{b_1}{a_p}=1,
\]
since $L_R(0)=b_1/a_p=1$.
So that, the similarity transform maps the ME representation into a standardized representation $(\bm{\alpha}, \bm{T})$.

The density and distribution function satisfy
\begin{align}
f_R(x) & = \bm{\alpha} \e^{\bm{T}x} \bm{t} = \bm{\beta} \e^{\bm{S}x} \bm{s}, \qquad x > 0,\label{eqnDensityOfME}\\
1-F_R(x) & = \int_x^\infty f_R(y) \dd y = \bm{\alpha} \e^{\bm{T}x} \bm{1} = \bm{\beta} \e^{\bm{S}x} (-\bm{S}^{-1})\bm{s}, \qquad x > 0,\label{eqnSurvivalOfME}
\end{align}
using that $(\bm{M}^{-1}\bm{S}\bm{M})^n=\bm{M}^{-1}\bm{S}^n\bm{M}$ for every $n\geq 0$.
The choice of $\bm{M}$ in \eqref{eq:M_matrix} is not unique. Any invertible matrix $\bm{M}$ satisfying $\bm{M}\bm{1}=-\bm{S}^{-1}\bm{s}$
yields a valid standardized ME representation with $\bm{t}=-\bm{T}\bm{1}$ and
$\bm{\alpha}\bm{1}=1$. Different choices of $\bm{M}$ produce different representations $(\bm{\alpha}, \bm{T})$, all describing the same distribution but with different matrix structures.

We conclude that any distribution with a rational Laplace-Stieltjes transform admits a standardized ME representation $(\bm{\alpha}, \bm{T})$ with $\bm{t} = -\bm{T}\bm{1}$ and $\bm{\alpha}\bm{1} = 1$.
While the construction above provides one explicit procedure, other transformations are possible. For the remainder of this paper, we work with such a minimal standardized representation without specifying the particular construction. We suppress the closing vector $\bm{t}$ in the notation, writing simply $(\bm{\alpha}, \bm{T})$, since $\bm{t} = -\bm{T}\bm{1}$ is determined by $\bm{T}$.

In particular, spectrally negative Lévy processes with rational jumps have matrix-exponential jump distributions, and the Laplace exponent of the \emph{unkilled} process takes the form
\begin{equation}\label{eq:psi_ME_final}
\psi(\theta) = d\theta + \frac{\sigma^2}{2}\theta^2 + \lambda\big(\bm{\alpha}(\theta \bm I - \bm{T})^{-1}\bm{t} - 1\big),
\end{equation}
where $(\bm{\alpha}, \bm{T})$ is a minimal standardized representation of the jump distribution. For the remainder of the paper, we work with spectrally negative Lévy processes of this form, where $(\bm{\alpha}, \bm{T})$ is of order $p$ with $\bm{t} = -\bm{T}\bm{1}$ and $\bm{\alpha}\bm{1} = 1$, and we incorporate killing at rate $q$ separately through the independent exponential time $e_q$. Equivalently,
\[
\EE\big[\e^{\theta X_t}\mathds{1}\{t<e_q\}\big]=\e^{(\psi(\theta)-q)t}.
\]

\section{Orbit Processes and RAP-Modulation}\label{sec:orbit_processes_and_rap_modulated_stochastic_fluid_processes}
The extension from phase-type to matrix-exponential distributions requires replacing the modulator Markov process of the PH distribution with a more general class of Markov processes. The appropriate framework is provided by orbit processes, which were introduced by \cite{AsmussenBladt1999} and further developed in \cite{bean2010quasi,BuchholzTelek2013}. Recent work by \cite{BeanNguyenNielsenPeralta2022} has extended this framework to study rational arrival processes (RAP) and RAP-modulated stochastic fluid processes with generalized orbit processes that can transition between multiple state spaces.

A \emph{jump} is any discontinuity of the orbit process $\bm{A}$ that does not correspond to termination, and the \emph{interarrival times} are the lengths of the intervals between successive jumps. A \emph{sojourn} is the time spent within a single component between either two consecutive jumps or between the last jump and termination, and its length is the \emph{sojourn time}. The \emph{residual sojourn time} at time $t$ is the remaining time in the current component from $t$ until the next jump or termination. We assume throughout that the relevant processes exhibit a finite number of jumps in any compact time interval, which guarantees that interarrival times are strictly positive \cite{AsmussenBladt1999}.

\subsection{From Markov Chains to Orbit Processes}\label{subsec:from_markov_to_orbits}

Understanding matrix-exponential distributions that are not phase-type in terms of finite-state Markov jump processes is currently limited. 
Nevertheless, they admit a realization through \emph{orbit processes}, a class of piecewise deterministic Markov processes introduced by \cite{AsmussenBladt1999}. An orbit process $\{\bm{A}_t\}_{t \ge 0}$ takes values in a state space $\mathcal{A} \subset \mathbb{R}^p$, which is required to be compact and such that $\bm{A}_{t}\bm{1} = 1$ for every $t\geq 0$. The triple $(\bm{\alpha}, \bm{C}, \bm{D})$ and the dynamics described below must be compatible with $\mathcal{A}$ in the sense that the process never exits it. In general, spelling out $\mathcal{A}$ explicitly is difficult, and one instead works under the assumption that a compact invariant state space exists for the given triple. In the phase-type case, $\mathcal{A}$ can be taken as a subset of the simplex $[0,1]^p$, where the orbit has the interpretation of a conditional distribution of a Markov chain given non-absorption.

Between jump epochs, the orbit process $\{\bm{A}_t\}_{t \ge 0}$ evolves deterministically according to the ordinary differential equation
\begin{equation}\label{eq:orbit_dynamics}
\frac{\dd \bm{A}_s}{\dd s} = \bm{A}_s \bm{C} - (\bm{A}_s \bm{C} \bm{1}) \bm{A}_s,
\end{equation}
where $\bm{C}$ is a $p \times p$ real matrix derived from the matrix-exponential representation, which satisfies certain conditions \cite{AsmussenBladt1999}. 
The solution to \eqref{eq:orbit_dynamics} is
\begin{equation}\label{eq:orbit_solution_classical}
\bm{A}_{t+r} = \frac{\bm{A}_t e^{\bm{C}r}}{\bm{A}_t e^{\bm{C}r} \bm{1}},
\end{equation}
which preserves the normalization $\bm{A}_{t+r}\bm{1} = 1$. The orbit traces out a deterministic curve in the state space $\mathcal{A}$. When $\bm{C} = \bm{T}$ is a phase-type subgenerator, the dynamics \eqref{eq:orbit_solution_classical} reduce to the conditional distribution of a Markov chain on $\{1, \ldots, p\}$ given non-absorption up to time $r$. 
For general matrix-exponential representations, $\bm{C}$ need not be a subgenerator, yet the orbit dynamics remain well-defined and produce a valid piecewise deterministic Markov process, provided $\mathcal{A}$ is invariant.

The rational arrival process $N$ generated by $\{\bm{A}_t\}$ records the jump epochs of the orbit. 
They occur at location-dependent intensities. 
We consider a rational arrival process \cite{AsmussenBladt1999}, which generalizes the Markovian arrival process \cite{neuts1979versatile} to allow matrix-exponential interarrival times. 
The RAP process is characterized by some parameters $(\bm{\alpha}, \bm{C},\bm{D})$, 
The jump intensity of a RAP process at state $\bm{a} \in \mathcal{A}$ is given by
\begin{equation}\label{eq:jump_intensity}
\lambda_N(\bm{a}): = -\bm{a} \bm{C} \bm{1}\geq 0,\qquad \bm a\in \mathcal{A}.
\end{equation}
The survival function of the residual sojourn time at time $t$, given $\bm{A}_t = \bm{a}$, is (see \cite[Equation (2.5)]{BeanNguyenNielsenPeralta2022})
\begin{equation}\label{eq:no_jump_probability}
\mathbb{P}(\text{residual sojourn time at } t > h \mid \bm{A}_t = \bm{a}) = \bm{a} e^{\bm{C}h} \bm{1} \ge 0, \qquad h \ge 0.
\end{equation}
The right-hand side is precisely the tail of a standardized ME distribution with parameters $(\bm{a}, \bm{C})$, as in \eqref{eqnSurvivalOfME}. This is the fundamental reason why rational arrival processes are the natural framework for modelling point processes with matrix-exponential interarrival times: the residual sojourn time automatically inherits a matrix-exponential distribution, regardless of the elapsed time $t$.

Given that a jump occurs at time $t$, the post-jump location is
\begin{equation}\label{eq:jump_location}
\bm{A}_{t} = \frac{\bm{A}_{t-} \bm{D}}{\bm{A}_{t-} \bm{D} \bm{1}}.
\end{equation}

In \cite{AsmussenBladt1999}, the processes considered were non-terminating, which means, they do not go to a cemetery state after a certain independent exponential clock rings. In such non-terminating case, we have the normalization condition
\begin{equation}\label{eq:normalization_cd}
\bm{C}\bm{1} + \bm{D}\bm{1} = \bm{0},
\end{equation}
from which the jump intensity can be rewritten as $\lambda(\bm{a}) = \bm{a} \bm{D} \bm{1}$.
In the next sections, we will relax the above condition to allow termination of the process.

Theorem 1.1 of \cite{AsmussenBladt1999} establishes that the joint density of the first $n$ interarrival times of the RAP process $N$ is
\begin{equation}\label{eq:interarrival_density}
f_{N,p}(x_1, \ldots, x_n) := \bm{\alpha} e^{\bm{C}x_1} \bm{D} e^{\bm{C}x_2} \bm{D} \cdots e^{\bm{C}x_n} \bm{t},
\end{equation}
where $\bm{t} := -\bm{C}\bm{1}=\bm{D}\bm{1}$ and $\bm{\alpha} := \bm{A}_0$ is the initial orbit state.
This representation parallels the Markovian arrival process formula (see, e.g., \cite[Theorem 10.2.3]{BladtNielsen2017}), with the orbit state $\bm{A}_t$ playing the role of the underlying Markov chain state in such a setting.
In particular, when $\bm{C} = \bm{T}$ and $\bm{D} = \bm{t}\bm{\alpha}$, the interarrival times become independent and identically distributed with matrix-exponential distribution $(\bm{\alpha}, \bm{T}, \bm{t})$, recovering the renewal process case.
In this case, $\mathcal{A}$ can be identified with the set of residual lifetime distributions of the interarrival distribution, for which a compact invariant state space satisfying the required conditions indeed exists \cite{AsmussenBladt1999}.

Crucially, the orbit representation accommodates matrix-exponential interarrival distributions that are not phase-type. While the matrix $\bm{C}$ in \eqref{eq:orbit_dynamics} may have negative off-diagonal entries and the matrix $\bm{D}$ may have negative entries, the algebraic structure encoded in \eqref{eq:interarrival_density} ensures that $f_{N,p}$ integrates to the correct probability mass. The orbit process thus provides a general framework for representing point processes with finite-dimensional conditional distributions, strictly generalizing the Markovian arrival process.

\subsection{Orbit Dynamics across Multiple Spaces}

The orbit processes we employ take values in a partitioned state space and can jump between distinct components corresponding to different dynamical regimes, until a possible termination time. We begin with the simplest case before describing the more general framework.

\subsubsection{Simple Two-Space Orbit Processes}\label{subsubsectionTwoSpaceOrbit}

Consider an orbit process $\{\bm{A}_t\}_{t \ge 0}$ taking values in a partitioned state space $\mathcal{Z} = \mathcal{Z}^+ \cup \mathcal{Z}^-$, where $\mathcal{Z}^+ \subset \mathbb{R}^{m^+}$ and $\mathcal{Z}^- \subset \mathbb{R}^{m^-}$ are disjoint compact subsets of the affine hyperplanes $\{\bm a\in\mathbb{R}^{m^+}:\bm a\bm 1=1\}$ and $\{\bm a\in\mathbb{R}^{m^-}:\bm a\bm 1=1\}$, respectively, each required to be invariant under the orbit dynamics within the corresponding region. Here $m^\pm\in \mathbb{N}$. The regions $\mathcal{Z}^+$ and $\mathcal{Z}^-$ will correspond to different regimes of an associated stochastic process.

Between jumps, the orbit evolves deterministically according to the ordinary differential equation
\begin{equation}\label{eq:orbit_two_space}
\frac{\dd \bm{A}_s}{\dd s} = \begin{cases}
\bm{A}_s \bm{C}^+ - (\bm{A}_s \bm{C}^+ \bm{1}) \bm{A}_s & \text{if } \bm{A}_s \in \mathcal{Z}^+, \\
\bm{A}_s \bm{C}^- - (\bm{A}_s \bm{C}^- \bm{1}) \bm{A}_s & \text{if } \bm{A}_s \in \mathcal{Z}^-,
\end{cases}
\end{equation}
where $\bm{C}^+ \in \mathbb{R}^{m^+ \times m^+}$ and $\bm{C}^- \in \mathbb{R}^{m^- \times m^-}$ govern the deterministic evolution within each region. The solution preserves the normalization $\bm{A}_{t+r}\bm{1}=1$:
\begin{equation}\label{eq:orbit_solution_two_space}
\bm{A}_{t+r} = \begin{cases}
{ \frac{\bm{A}_t e^{\bm{C}^+ r}}{\bm{A}_t e^{\bm{C}^+ r} \bm{1}}} \in \mathcal{Z}^+ & \text{if } \bm{A}_t \in \mathcal{Z}^+, \\
\frac{\bm{A}_t e^{\bm{C}^- r}}{\bm{A}_t e^{\bm{C}^- r} \bm{1}} \in \mathcal{Z}^- & \text{if } \bm{A}_t \in \mathcal{Z}^-.
\end{cases}
\end{equation}

Jumps between regions occur according to location-dependent intensities. We allow for terminating orbits by relaxing the normalization conditions to permit
\begin{equation}\label{eq:normalization_two_space_terminating}
\bm{C}^+ \bm{1} + \bm{D}^{+-} \bm{1} \le \bm{0}, \qquad \bm{C}^- \bm{1} + \bm{D}^{-+} \bm{1} \le \bm{0},
\end{equation}
with the \emph{deficits} $-(\bm{C}^+ \bm{1} + \bm{D}^{+-} \bm{1})\ge \bm{0}$ and $-(\bm{C}^- \bm{1} + \bm{D}^{-+} \bm{1})\ge \bm{0}$ understood as the \emph{termination intensity} from region $\mathcal{Z}^+$ and $\mathcal{Z}^-$, respectively. 
While this terminating framework differs from the non-terminating setup emphasized in \cite{BeanNguyenNielsenPeralta2022}, the algebraic structure of their analysis accommodates termination without mathematical difficulty.

When $\bm{A}_t \in \mathcal{Z}^+$, the orbit jumps to $\mathcal{Z}^-$ with intensity $\lambda_N(\bm{A}_t) = \bm{A}_t \bm{D}^{+-} \bm{1}$.
We assume the latter is nonnegative for all states $\bm{A}_t$. 
Similarly, jumps from $\mathcal{Z}^-$ to $\mathcal{Z}^+$ occur with intensity $\lambda_N(\bm{A}_t) = \bm{A}_t \bm{D}^{-+} \bm{1}$. 
Defining the \emph{sojourn intensity} $\lambda_S$ of leaving region $\mathcal{Z}^k$ with $k\in\{+,-\}$, either by jumping to the other region or by termination, then $\lambda_S(\bm{A}_t) =-\bm{A}_t(\bm{C}^k \bm{1} + \bm{D}^{k\ell} \bm{1})+ \bm{A}_t\bm{D}^{k\ell} \bm{1}= -\bm{A}_t \bm{C}^k \bm{1}\ge 0$, where $\ell\in\{+,- \}$ and $\ell\neq k$.
Consequently, the survival function of the residual sojourn time at time $t$, given $\bm{A}_t = \bm{a} \in \mathcal{Z}^k$, is
\begin{equation}
\mathbb{P}(\text{residual sojourn time at } t > h \mid \bm{A}_t = \bm{a}) = \bm{a} e^{\bm{C}^k h} \bm{1} \ge 0, \qquad h \ge 0.\label{eq:survival-residual-simple}
\end{equation}

Given that a jump occurs at time $t$, the post-jump location is deterministic:
\begin{equation}\label{eq:jump_two_space}
\bm{A}_t = \begin{cases}
\frac{\bm{A}_{t-} \bm{D}^{+-}}{\bm{A}_{t-} \bm{D}^{+-} \bm{1}} \in \mathcal{Z}^- & \text{if } \bm{A}_{t-} \in \mathcal{Z}^+, \\
\frac{\bm{A}_{t-} \bm{D}^{-+}}{\bm{A}_{t-} \bm{D}^{-+} \bm{1}} \in \mathcal{Z}^+ & \text{if } \bm{A}_{t-} \in \mathcal{Z}^-.
\end{cases}
\end{equation}
The matrices $(\bm{C}^+, \bm{C}^-, \bm{D}^{+-}, \bm{D}^{-+})$, along with an initial orbit point, completely characterize the orbit process dynamics, encoding both the deterministic evolution within each region and the stochastic transitions between regions.

\subsubsection{Extension to Multiple Hyperplanes}
In more complex applications, a refined orbit structure with additional state space partitions is required. The framework of \cite{BeanNguyenNielsenPeralta2022} extends the two-space construction by partitioning each region $\mathcal{Z}^k$, $k\in \{+,-\}$, into multiple components. Specifically, each region decomposes as
\begin{equation}\label{eq:hyperplane_partition_general}
\mathcal{Z}^k = \bigcup_{i=1}^{n_k} \mathcal{Z}_i^k, \qquad k \in \{+, -\},
\end{equation}
where each $\mathcal{Z}_i^k$ is a compact subset of a distinct affine hyperplane, embedded in a common ambient space $\mathbb{R}^{\sum_i m_i^k}$. More precisely, $\mathcal{Z}_i^k$ consists of points of the form $(\bm{0}, \ldots, \bm{0}, \bm{a}, \bm{0}, \ldots, \bm{0})$, where $\bm{a}\in \mathbb{R}^{m_i^k}$ occupies the $i$-th block and satisfies $\bm{a}\bm{1} = 1$, with all other blocks equal to zero. The hyperplanes corresponding to distinct indices $i$ are mutually orthogonal, with the sets $\mathcal{Z}_i^k$ geometrically separated within $\mathbb{R}^{\sum_i m_i^k}$.
Within each component $\mathcal{Z}_i^k$, the orbit evolves according to a matrix $\bm{C}_{ii}^k$, and jumps between components within the same region or across regions are governed by a collection of matrices $\{\bm{C}_{ij}^k, \bm{D}_{ij}^{k\ell}\}$, where $i\in \{1,\ldots, n_k\}$, $j\in \{1,\ldots, n_\ell\}$ and $k,\ell\in \{+,-\}$ with $k\neq\ell$. 
Explicitly, the diagonal blocks $\bm{C}_{ii}^k$ govern the deterministic orbit evolution within $\mathcal{Z}_i^k$, the off-diagonal blocks $\bm{C}_{ij}^k$ encode jumps within region $\mathcal{Z}^k$, and the blocks $\bm{D}_{ij}^{k\ell}$ specify jumps from $\mathcal{Z}_i^k$ to $\mathcal{Z}_j^\ell$.
We emphasize that the matrices are defined in such a way that the evolution of the orbit process is well-defined, that is, once it jumps to a new component, it takes values there until the sojourn time.
The survival function of the residual sojourn time at time $t$, given $\bm{A}_t = (\bm{0}, \ldots, \bm{0}, \bm{a}, \bm{0}, \ldots, \bm{0}) \in \mathcal{Z}_i^k$, is
\begin{equation}\label{eq:residual_sojourn_component}
\mathbb{P}(\text{residual sojourn time at } t > h \mid \bm{A}_t = (\bm{0}, \ldots, \bm{0}, \bm{a}, \bm{0}, \ldots, \bm{0})) = \bm{a} e^{\bm{C}_{ii}^k h} \bm{1} \ge 0, \qquad h \ge 0.
\end{equation}
See \cite[Section 3.1]{BeanNguyenNielsenPeralta2022} for details.

The normalization conditions are also relaxed when compared with \cite{BeanNguyenNielsenPeralta2022} to allow termination:
\begin{equation}\label{eq:normalization_hyperplane}
\sum_{j=1}^{n_k} \bm{C}_{ij}^k\bm{1} + \sum_{j=1}^{n_\ell} \bm{D}_{ij}^{k\ell}\bm{1} \le \bm{0},\qquad \text{where }k\neq \ell,\; i\in \{1,\ldots, n_k\},
\end{equation}
with any deficit, that is, the negative of \eqref{eq:normalization_hyperplane}, interpreted as the termination intensity from hyperplane $\mathcal{Z}_i^k$. The algebraic structure preserves the essential features of the simple two-space case while accommodating additional complexity.

In Section~\ref{sec:main_result_scale_function_representation}, we employ both frameworks. For simpler cases, the two-space orbit process described above suffices. For more complex situations, we construct specific orbit processes with hyperplane structure tailored to the problem at hand. The full technical details of the general hyperplane framework are not needed for our purposes, as the specific constructions we employ will be presented explicitly when required.

\subsection{RAP-Modulated Stochastic Fluid Processes and First Passage Probabilities}\label{subsectionRAPMSFP}
A RAP-modulated stochastic fluid process is a Markov additive process $\{(R_t, \bm{A}_t)\}_{t \ge 0}$ where the level component $R_t$ is piecewise linear and the modulator component $\bm{A}_t$ is an orbit process \citep{BeanNguyenNielsenPeralta2022}. These processes generalize classical fluid flow models, in which the modulator is a finite-state Markov chain, to accommodate orbit processes and thus matrix-exponential structures.
A fluid flow process $(R_t,A_t)_{t\geq 0}$ is a continuous-time stochastic process whose level $R$ evolves linearly and deterministically between random switching times determined by the Markov chain $A$ on a finite state space. The reader can think of it as a water storage which at random times, is being filled or drained at linear speed.

The state space partition $\mathcal{Z} = \mathcal{Z}^+ \cup \mathcal{Z}^- \cup \mathcal{Z}^0$ of the RAP-modulated stochastic fluid process, corresponds to increasing, decreasing, and constant regimes for the level process. In the simple two-space case $\mathcal{Z} = \mathcal{Z}^+ \cup \mathcal{Z}^-$ that we will consider, the level process takes the form
\begin{equation}\label{eq:level_process}
R_t = \int_0^t \big(\mathds{1}\{\bm{A}_s \in \mathcal{Z}^+\} - \mathds{1}\{\bm{A}_s \in \mathcal{Z}^-\}\big)\,\dd s, \qquad t \ge 0,
\end{equation}
so that $R_t$ increases at unit rate when the orbit is in $\mathcal{Z}^+$ and decreases at unit rate when it is in $\mathcal{Z}^-$. More generally, when each region $\mathcal{Z}^k$ admits a hyperplane decomposition $\mathcal{Z}^k = \bigcup_{i=1}^{n_k} \mathcal{Z}_i^k$, the orbit dynamics within hyperplane $\mathcal{Z}_i^k$ are governed by the matrix $\bm{C}_{ii}^k$, while jumps within and across regions are governed by the off-diagonal blocks $\bm{C}_{ij}^k$ and $\bm{D}_{ij}^{k\ell}$, respectively. These matrices are assembled into block matrices $\bm{C}^k$ and $\bm{D}^{k\ell}$ as in \eqref{eq:block_matrices} below.

A fundamental result in the orbit framework is the characterization of first passage probabilities via the matrix $\bm{\Psi}$ \citep{BeanNguyenNielsenPeralta2022}. Consider a RAP-modulated stochastic fluid process $\{(R_t, \bm{A}_t)\}_{t \ge 0}$ starting with $R_0 = 0$ and $\bm{A}_0 = \bm{\alpha} \in \mathcal{Z}^+$, and let $\mathbb{E}_{\bm{\alpha}}$ denote the expectation conditional on this initial condition. Let $\tau_- = \inf\{t > 0 : R_t \le 0\}$ denote the first return time to level zero. The matrix $\bm{\Psi}$ is characterized by the identity
\begin{equation}\label{eq:psi_interpretation}
    \mathbb{E}_{\bm{\alpha}}[\bm{A}_{\tau_-} \mathds{1}\{\tau_- < \infty\}] = \bm{\alpha} \bm{\Psi},
\end{equation}
which provides the direct physical interpretation of $\bm{\Psi}$ as the mean value of the orbit process evaluated at the first downcrossing epoch.
The importance of this identity will be evident in Section \ref{sec:main_result_scale_function_representation}, when we relate the first hitting time of level $-x$ by the L\'evy process with the mean value of the orbit at $\tau_-$.

The computation of $\bm{\Psi}$ relies on a recursive sequence of matrices $\{\bm{\Psi}_n\}_{n \ge 0}$ satisfying Sylvester's recurrence
\begin{equation}\label{eq:psi_sylvester}
\bm{C}^+\bm{\Psi}_n + \bm{\Psi}_n \bm{C}^- 
= -\bm{D}^{+-} - \bm{\Psi}_{n-1}\bm{D}^{-+}\bm{\Psi}_{n-1},
\qquad n\ge 1,
\end{equation}
with $\bm{\Psi}_0 = \bm{0}$. 
As $n \to \infty$, the sequence $\{\bm{\Psi}_n\}_{n\ge 0}$ converges entrywise to $\bm{\Psi}$ which satisfies the identity \eqref{eq:psi_interpretation} (see \cite[Section 4.1]{BeanNguyenNielsenPeralta2022}).
The precise physical interpretation of each $\bm{\Psi}_n$ in terms of restricted path sets is given in \cite[Theorem 4.1]{BeanNguyenNielsenPeralta2022}.
Here, $\bm{C}^+$, $\bm{C}^-$, $\bm{D}^{+-}$ and $\bm{D}^{-+}$ are block matrices assembled from their corresponding subblocks:
\begin{equation}\label{eq:block_matrices}
\bm{C}^+ = 
\begin{pmatrix}
\bm{C}_{11}^+ & \cdots & \bm{C}_{1,n_+}^+\\
\vdots & \ddots & \vdots\\
\bm{C}_{n_+,1}^+ & \cdots & \bm{C}_{n_+,n_+}^+
\end{pmatrix},
\qquad
\bm{D}^{+-} = 
\begin{pmatrix}
\bm{D}_{11}^{+-} & \cdots & \bm{D}_{1,n_-}^{+-}\\
\vdots & \ddots & \vdots\\
\bm{D}_{n_+,1}^{+-} & \cdots & \bm{D}_{n_+,n_-}^{+-}
\end{pmatrix},
\end{equation}
and similarly for $\bm{C}^-$ and $\bm{D}^{-+}$. 

The quadratic term $\bm{\Psi}_{n-1}\bm{D}^{-+}\bm{\Psi}_{n-1}$ captures the contribution from paths that make an excursion from $\mathcal{Z}^+$ to $\mathcal{Z}^-$ and back to $\mathcal{Z}^+$ before eventually reaching level zero, with the orbit values at the two transitions from $\mathcal{Z}^-$ to $\mathcal{Z}^+$ contributing the two $\bm{\Psi}_{n-1}$ factors. The limit matrix $\bm{\Psi}$ therefore satisfies the nonsymmetric algebraic Riccati equation \citep[Corollary 4.2]{BeanNguyenNielsenPeralta2022}:
\begin{equation}\label{eq:riccati}
    \bm{C}^+ \bm{\Psi} + \bm{\Psi} \bm{C}^- = -\bm{D}^{+-} - \bm{\Psi} \bm{D}^{-+} \bm{\Psi}.
\end{equation}
We remark that establishing the uniqueness of solutions to \eqref{eq:riccati} in the general matrix-exponential setting remains an open problem. The constructions in Theorems \ref{teoScaleFnBV} and \ref{teoScaleFnUV} rely instead on the limit of the recursive sequence \eqref{eq:psi_sylvester}, which selects a distinguished fixed point as the limit of the recursion started from $\bm{\Psi}_0=\bm{0}$. We nonetheless employ \eqref{eq:riccati} to characterize the limiting parameters and the scale function of the L\'evy process.

\subsubsection{The Downward Record Process}

For applications to scale functions, we require the orbit distribution at downcrossing times from arbitrary levels. The latter are the analogue of the \emph{descending ladder variables} of random walks (see \cite{MR0270403} Chapter XII.1); and the ladder process of L\'evy processes (see \cite{MR1406564} Chapter VI). Note that in this case, we also need to encode the values of the orbit process. 
Following \cite[Section 4.2]{BeanNguyenNielsenPeralta2022}, we introduce the \emph{downward record process}.

For $x \ge 0$, define $\gamma_{\{-x\}} := \inf\{t>0 : R_t=-x\}$ as the first time the level process downcrosses level $-x \le 0$. 
Note that on the event $R_t=-x$ for the first time, we have $\bm{A}_t\in\mathcal{Z}^-$. 
The downward record process $\{(\ell_x, \bm{O}_x)\}_{x \ge 0}$ is defined by
\begin{equation}\label{eq:downward_record}
(\ell_x, \bm{O}_x) = \begin{cases}
(R_{\gamma_{\{-x\}}}, \bm{A}_{\gamma_{\{-x\}}}) & \text{if } \gamma_{\{-x\}} < \infty, \\
(\infty, \Delta) & \text{if } \gamma_{\{-x\}} = \infty,
\end{cases}
\end{equation}
where $\Delta$ is an isolated cemetery state. When $\bm{O}_x \neq \Delta$, the vector $\bm{O}_x$ represents the orbit value at the first downcrossing of level $-x$. The process $\{\bm{O}_x\}_{x \ge 0}$ is a (possibly terminating) concatenation of orbits with state space $\mathcal{Z}^- \cup \{\Delta\}$.

The physical interpretation of $\bm{O}_x$ is that it tracks the ``state'' of the system at each new low of the level process. As $x$ increases, $\bm{O}_x$ records successive orbit values at downcrossing epochs. This sequence provides complete information about the orbit trajectory during the excursion below level zero.

One important property of $\bm{O}_x$, is that its expected value has an explicit matrix exponential form. Indeed, if the process $\{(R_t, \bm{A}_t)\}_{t \ge 0}$ starts with orbit $\bm{\beta} \in \mathcal{Z}^-$, then \cite[Theorem 4.4]{BeanNguyenNielsenPeralta2022} gives
\begin{equation}\label{eq:O_x_formula_minus}
\mathbb{E}_{\bm{\beta}}[\bm{O}_x \mathds{1}\{\gamma_{\{-x\}} < \infty\}] = \bm{\beta} \, e^{(\bm{C}^- + \bm{D}^{-+}\bm{\Psi})x}.
\end{equation}
If instead the process starts with orbit $\bm{\alpha} \in \mathcal{Z}^+$, then the first downcrossing necessarily passes through $\mathcal{Z}^-$ via the matrix $\bm{\Psi}$, and \cite[Corollary 4.5]{BeanNguyenNielsenPeralta2022} establishes
\begin{equation}\label{eq:O_x_formula}
\mathbb{E}_{\bm{\alpha}}[\bm{O}_x \mathds{1}\{\gamma_{\{-x\}} < \infty\}] = \bm{\alpha} \bm{\Psi} \, e^{(\bm{C}^- + \bm{D}^{-+}\bm{\Psi})x}.
\end{equation}
The probability of downcrossing level $-x$ is obtained in either case by applying the sum vector. For example, when starting in $\bm{\alpha} \in \mathcal{Z}^+$, we have:
\begin{equation}\label{eq:downcrossing_prob}
\mathbb{P}_{\bm{\alpha}}(\gamma_{\{-x\}} < \infty) = \bm{\alpha} \bm{\Psi} \, e^{(\bm{C}^- + \bm{D}^{-+}\bm{\Psi})x} \bm{1}.
\end{equation}

The formula \eqref{eq:O_x_formula} decomposes naturally: the factor $\bm{\alpha} \bm{\Psi}$ represents the expected orbit value at the first entry into $\mathcal{Z}^-$ (i.e., at downcrossing of level zero) by \eqref{eq:psi_interpretation}, while the exponential term $e^{(\bm{C}^- + \bm{D}^{-+}\bm{\Psi})x}$ propagates this distribution forward as the level continues to decrease by an additional amount $x$. The matrix $\bm{C}^- + \bm{D}^{-+}\bm{\Psi}$ captures the combined effect of deterministic orbit evolution (via $\bm{C}^-$) and the influence of potential returns to $\mathcal{Z}^+$ and subsequent re-entries (encoded in $\bm{D}^{-+}\bm{\Psi}$).

This construction is essential for computing scale functions at positive levels, as it provides the link between the orbit dynamics and the probability of reaching arbitrary negative levels starting from above zero.

\begin{remark}\label{rem:terminating_raps}
The results from \cite[Sections 4.1 and 4.2]{BeanNguyenNielsenPeralta2022} cited above, including the characterization of $\bm{\Psi}$ via \eqref{eq:psi_sylvester} and the downward record process formulas \eqref{eq:O_x_formula_minus}--\eqref{eq:downcrossing_prob}, were established for non-terminating RAP-modulated fluid processes. However, these results remain valid in the terminating case with relaxed normalization conditions \eqref{eq:normalization_two_space_terminating} or \eqref{eq:normalization_hyperplane}. The key observation is that all formulas in those sections compute expectations conditional on downcrossing events within $\mathcal{Z}^+$ and $\mathcal{Z}^-$. Allowing termination simply means that paths which terminate before the relevant downcrossing event occurs do not contribute to these conditional expectations. The proofs from \cite{BeanNguyenNielsenPeralta2022} extend directly to the terminating case, and the verification of this extension is straightforward.
\end{remark}

\begin{example}\label{ex:ME_renewal}
Consider a RAP-modulated fluid process $\{(R_t, \bm{A}_t)\}_{t \ge 0}$ in which the sojourn time within each component $\mathcal{Z}_i^k$ has a matrix-exponential distribution with representation $(\bm{\alpha}_i^k, \bm{S}_i^k)$, and transitions between components are governed by a substochastic matrix $\bm{P} = \{p_{ij}^{k\ell}\}$ with
\begin{equation}
\sum_{\substack{j=1\\ j\neq i}}^{n_k} p_{ij}^{kk} + \sum_{j=1}^{n_\ell} p_{ij}^{k\ell} \le 1, \qquad \text{where }k\neq\ell, \quad i \le n_k,
\end{equation}
where the deficit $1 - \sum_{j\neq i} p_{ij}^{kk} - \sum_{j=1}^{n_\ell} p_{ij}^{k\ell}$ represents the termination probability from component $\mathcal{Z}_i^k$. Following \cite[Example 3.2]{BeanNguyenNielsenPeralta2022}, the subblocks are chosen as
\begin{equation}\label{eq:ME_renewal_params}
\bm{C}_{ii}^k = \bm{S}_i^k, \qquad \bm{C}_{ij}^k = p_{ij}^{kk}(-\bm{S}_i^k\bm{1})\,\bm{\alpha}_j^k, \qquad \bm{D}_{ij}^{k\ell} = p_{ij}^{k\ell}(-\bm{S}_i^k\bm{1})\,\bm{\alpha}_j^\ell,
\end{equation}
for $i \neq j$ in the second expression and $\ell \neq k$ in the third. With these choices, the sojourn time in $\mathcal{Z}_i^k$ starting from orbit value $\bm{\alpha}_i^k$ has survival function $\bm{\alpha}_i^k e^{\bm{S}_i^k t}\bm{1}$, which coincides with the tail of the standardized ME distribution of parameters $(\bm{\alpha}_i^k, \bm{S}_i^k)$ (see \eqref{eqnSurvivalOfME}). This follows directly from \eqref{eq:residual_sojourn_component}, since $\bm{C}_{ii}^k = \bm{S}_i^k$ and the orbit starts fresh at $\bm{\alpha}_i^k$ upon entering $\mathcal{Z}_i^k$, so that the residual sojourn time coincides in distribution with the full sojourn time.
Upon completion of the sojourn, the orbit exits $\mathcal{Z}_i^k$ and lands at $\bm{\alpha}_j^\ell$ with probability $p_{ij}^{k\ell}$, or terminates with probability $1 - \sum_{j\neq i} p_{ij}^{kk} - \sum_{j} p_{ij}^{k\ell}$.
From \eqref{eq:normalization_hyperplane}, the termination intensity from $\mathcal{Z}_i^k$ is
\[
-\sum_{j=1}^{n_k} \bm{C}_{ij}^k\bm{1} - \sum_{j=1}^{n_\ell} \bm{D}_{ij}^{k\ell}\bm{1} = (-\bm{S}_i^k\bm{1})\left(1-\sum_{\substack{j=1\\ j\neq i}}^{n_k} p_{ij}^{kk}- \sum_{j=1}^{n_\ell}p_{ij}^{k\ell}\right),
\]
where $k\neq\ell$. When $\bm{P}$ is stochastic, the resulting process is a Markov renewal process with matrix-exponential sojourn times and transition kernel $\bm{P}$; when $\bm{P}$ is substochastic, termination may occur at the completion of any sojourn.
\end{example}

\section{Path Decomposition of the L\'evy Process and Embedding}\label{sec:main_result_scale_function_representation}

\subsection{Path Decomposition at Level $-x$}\label{sec:path_decomp}
A spectrally negative Lévy process has no positive jumps, so any \emph{upcrossing} of a level, that is, a time in which it hits or surpasses a positive level, occurs continuously.
A \emph{downcrossing} may occur either continuously (via a Brownian component when $\sigma > 0$, or via negative drift when $\sigma = 0$ and $d < 0$) or by a negative jump. This yields a natural two-stage decomposition of the event $\{\tau_{\{-x\}} < e_q\}$ (illustrated in Figure~\ref{fig:path_decomp} for the case $\sigma = 0$, $d > 0$): 
\begin{enumerate}
\item \textbf{First downcrossing of $-x$.} Starting from $X_0 = 0$, the process must first pass below $-x$ at time $\tau^-_{\{-x\}} = \inf\{t : X_t < -x\}$.
Define the overshoot $Z := -x - X_{\tau^-_{\{-x\}}} \ge 0$, which measures how far below $-x$ the process lands at the first downcrossing. 
Note that $X$ cannot hit level $-x$ exactly upon a jump almost surely, as demonstrated below.
Therefore, if the downcrossing occurs by a jump, then $Z > 0$; if it occurs continuously, then $Z = 0$.

\item \textbf{Return to $-x$ from below.} From $-x - Z$ with $Z \ge 0$, the process must reach $-x$ by a continuous upcrossing, since there are no positive jumps. 
By the strong Markov property at $\tau^-_{\{-x\}}$ and the memoryless property of $e_q$, the first hitting time probability takes the form $\e^{-\Phi_q Z}$, by the classical ladder height identity for spectrally negative Lévy processes \cite[Theorem 3.12]{Kyprianou2014}. 
When the downcrossing is continuous ($Z = 0$), the return to $-x$ is immediate, since the sample paths of the Brownian component are regular for the lower and upper half-lines $(-\infty, 0), ( 0,\infty)$ (see \cite{MR0242261} or \cite[Lemma 6]{MR4130409}).

\end{enumerate}

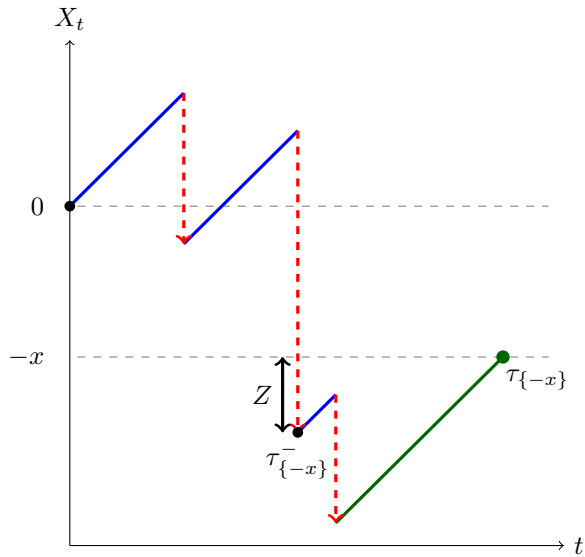
\begin{figure}[ht]
\centering
\begin{tikzpicture}[scale=1.0]
\draw[->] (0,-4.5) -- (0,2.2) node[above]{\small $X_t$};
\draw[->] (0,-4.5) -- (6.5,-4.5) node[right]{\small $t$};

\draw[dashed, gray] (0.1, 0) -- (6.3, 0);
\node at (-0.2, 0) [left]{\small $0$};
\draw[dashed, gray] (0.1, -2) -- (6.3, -2);
\node at (-0.2, -2) [left]{\small $-x$};

\draw[very thick, blue] (0, 0) -- (1.5, 1.5);
\draw[very thick, blue] (1.5, -0.5) -- (3.0, 1.0);
\draw[very thick, blue] (3.0, -3.0) -- (3.5, -2.5);

\draw[very thick, green!40!black] (3.5, -4.2) -- (5.7, -2.0);

\draw[very thick, red, dashed, ->] (1.5, 1.5) -- (1.5, -0.5);
\draw[very thick, red, dashed, ->] (3.0, 1.0) -- (3.0, -3.0);
\draw[very thick, red, dashed, ->] (3.5, -2.5) -- (3.5, -4.2);

\fill (0, 0) circle (2pt);
\node at (0.3, 0.3) [above right]{};

\fill[green!40!black] (5.7, -2.0) circle (2.5pt);
\node at (5.6, -2.3) [right]{\small $\tau_{\{-x\}}$};

\draw[<->, very thick] (2.8, -2.0) -- (2.8, -3.0);
\node at (2.8, -2.5) [left]{\small $Z$};

\fill (3.0, -3.0) circle (2pt);
\node at (3.0, -3) [below]{\small $\tau^-_{\{-x\}}$};


\node at (4.6, -3.2) [above]{};
\end{tikzpicture}
\caption{Path decomposition for hitting $-x$ when $\sigma = 0$, $d > 0$. The process starts at zero, drifts upward (blue) with downward jumps (red). First downcrossing $\tau^-_{\{-x\}}$ has overshoot $Z$, followed by continuous upcrossing (green) to reach $\tau_{\{-x\}}$.}
\label{fig:path_decomp}
\end{figure}

Note that if $\{\tau_n\}_{n \ge 1}$ are the jump epochs of $X$, then for any fixed level $-x < 0$, we have $\mathbb{P}(\exists\, n\in \mathbb{N}: X_{\tau_n} = -x)=0$. Indeed, for each $n\geq 1$,
\[
\PP(X_{\tau_n}=-x)=\EE\Big[\PP\big(C_n=x+X_{\tau_n-}\mid X_{\tau_n-}\big)\Big]=0,
\]
because $C_n$ is independent of the pre-jump history and has a density. A union bound over $n\in\mathbb{N}$ then yields the claim.

From the previous path decomposition, the hitting probability factors as
\begin{equation}\label{eq:hitting_factor}
\PP(\tau_{\{-x\}} < e_q) = \mathbb{E}\Big[\mathds{1}\big\{\tau^-_{\{-x\}} < e_q\big\}\mathbb{E}_{-Z}\big[\mathds{1}\{\tau_{\{0\}} < e_q\}\big]\Big] = \mathbb{E}\Big[\e^{-\Phi_q Z}\;\mathds{1}\big\{\tau^-_{\{-x\}} < e_q\big\}\Big].
\end{equation}
By the spatial homogeneity of Lévy processes, the probability of reaching $0$ from $-Z$ is identical to reaching $-Z$ from $0$. Thus, evaluating the conditional law $\mathbb{E}_{-Z}[\mathds{1}\{\tau_{\{0\}} < e_q\}]$ yields $\e^{-\Phi_q Z}$. We thereby translate the hitting time problem precisely into evaluating the transform of the overshoot (for general quintuple laws regarding the overshoot, see \cite[Theorem 7.7]{Kyprianou2014}).

To express this expectation in terms of the orbit framework, we will construct specific RAP-modulated fluid process embeddings for the bounded-variation case ($\sigma = 0$) and the unbounded-variation case ($\sigma > 0$) in the following subsections. The construction of these embeddings is more delicate than in the classical phase-type case analyzed by \cite{Ivanovs2021}, as we no longer have the convenient continuous-time Markov chain interpretations that provide direct probabilistic guidance for the embedding structure.


\subsection{Bounded Variation Case: $\sigma = 0$}

Recalling the L\'evy process description in \eqref{eqnXTrajectories}, assume $\sigma = 0$ and $d > 0$, so that $X$ has bounded variation with positive drift $d$, downward jumps of rate $\lambda$, and jump sizes following a matrix-exponential distribution with parameters $(\bm{\alpha}, \bm{T})$ of order $p$. We embed $X$ killed at rate $q$ into a RAP-modulated stochastic fluid process $\{(R_t, \bm{A}_t)\}_{t \ge 0}$ with $\pm 1$ rewards as illustrated in Figure~\ref{fig:embedding_bv}.

\subsubsection{Embedding}

The orbit process that we will use will be in the setting of Subsection \ref{subsubsectionTwoSpaceOrbit}, that is, a two-space orbit. 
The state space is $\mathcal{Z} = \mathcal{Z}^+ \cup \mathcal{Z}^-$, where $\mathcal{Z}^+ = \{1\} \subset \mathbb{R}$ is a singleton representing the filling regime and $\mathcal{Z}^-$ is an $p$-dimensional orbit space representing the drain regime. 
Formally, $\mathcal{Z}^-$ represents the relevant compact orbit state space within the affine hyperplane $\{\bm{a}\in\mathbb{R}^p:\bm{a}\bm{1}=1\}$ supporting the matrix-exponential parameters. The level process increases at rate $+1$ when $\bm{A}_t \in \mathcal{Z}^+$ and decreases at rate $-1$ when $\bm{A}_t \in \mathcal{Z}^-$.

To match the increment structure of $X$ with $\pm 1$ drifts, the sojourn rates in $\mathcal{Z}^+$ must be rescaled by $1/d$.
We use the parameters of Example \ref{ex:ME_renewal} as follows. Define $n_+=n_-=1$, $\bm{\alpha}^+_1=1$, $\bm{\alpha}^-_1=\bm{\alpha}$, $\bm{S}^+_1=-(\lambda + q)/d$ and $\bm{S}^-_1=\bm{T}$. A jump from $\mathcal{Z}^+$ occurs with probability $\lambda/(\lambda+q)=p^{+-}_{11}$, the probability that a Poisson arrival precedes the exponential killing time. After a sojourn in $\mathcal{Z}^-$ the process returns to $\mathcal{Z}^+$ almost surely, so $p^{-+}_{11}=1$. 
The orbit parameters are therefore
\begin{equation}
\begin{split}\label{eq:params_bv_plus}
\bm{C}^+ = -\frac{\lambda + q}{d}, \qquad &\bm{D}^{+-}= p_{11}^{+-}(-\bm{S}_1^+\bm{1})\,\bm{\alpha}_1^-= \frac{\lambda}{d}\,\bm{\alpha},\\
\bm{C}^- = \bm{T},\qquad \ \ \qquad &\bm{D}^{-+} =p_{11}^{-+}(-\bm{T}\bm{1})1= \bm{t} := -\bm{T}\bm{1}.
\end{split}
\end{equation}

\begin{remark}
When $q>0$, the normalization condition for $\mathcal{Z}^+$ holds with strict inequality: $C^+\bm{1} + \bm{D}^{+-}\bm{1} = -q/d < 0$, with the deficit representing termination at rate $q/d$ from $\mathcal{Z}^+$. For $\mathcal{Z}^-$, the normalization holds with equality: $\bm{C}^-\bm{1} + \bm{D}^{-+}\bm{1} = \bm{0}$, so there is no termination from the drain regime. This implements the $q$-killing via the termination mechanism of Section~\ref{subsectionRAPMSFP}, avoiding the need for an explicit cemetery state. When $q=0$ there is no killing, the normalization for $\mathcal{Z}^+$ holds with equality, and the orbit process is non-terminating. The killing therefore occurs only during the upward drift phase, that is, while the orbit is in $\mathcal{Z}^+$. This is consistent with the $q$-killing of $X$: since $X$ is killed at an independent exponential time with rate $q$, and the embedding associates the drift intervals of $X$ with sojourns in $\mathcal{Z}^+$, the killing is absorbed into the termination intensity of $\mathcal{Z}^+$ rather than $\mathcal{Z}^-$.
\end{remark}

A sojourn in $\mathcal{Z}^+$ has duration $\mathrm{Exp}((\lambda+q)/d)$, giving a level increment distributed as $d \cdot \mathrm{Exp}(\lambda+q)$, which matches in distribution, the time until $X$ has a negative jump or is killed. A sojourn in $\mathcal{Z}^-$ has a ME-distributed duration of parameters $(\bm{\alpha}, \bm{T})$, which matches the distribution of the negative jump of $X$. Upon exit from $\mathcal{Z}^+$, the process either enters $\mathcal{Z}^-$ with probability $\lambda/(\lambda+q)$ (a jump occurs) or terminates with probability $q/(\lambda+q)$ (killed).

\begin{figure}[ht]
\centering
\begin{tikzpicture}[scale=0.95]
\draw[->] (0,0) -- (3.8,0) node[right]{$t$};
\draw[->] (0,0) -- (0,4.2) node[above]{$X_t$};

\draw[very thick, blue] (0,0) -- (1,2);
\draw[very thick, blue] (1,1) -- (2,3);
\draw[very thick, blue] (2,1.5) -- (3,3.5);

\draw[very thick, red, dashed] (1,2) -- (1,1);
\draw[very thick, red, dashed] (2,3) -- (2,1.5);

\draw[dotted, gray] (0,2) -- (1,2);
\draw[dotted, gray] (0,1) -- (1,1);

\node at (1,2)[circle,fill,inner sep=1.5pt, blue]{};
\node at (1,1)[circle,fill,inner sep=1.5pt, red]{};

\node at (1.9,-1.2) {\small\textbf{(a) Spectrally negative L\'evy process}};

\begin{scope}[xshift=5.5cm]
\draw[->] (0,0) -- (9.2,0) node[right]{$t$};
\draw[->] (0,0) -- (0,4.2) node[above]{$R_t$};

\draw[very thick, blue] (0,0) -- (2,2);
\draw[very thick, blue] (3,1) -- (5,3);
\draw[very thick, blue] (6.5,1.5) -- (8.5,3.5);

\draw[very thick, red] (2,2) -- (3,1);
\draw[very thick, red] (5,3) -- (6.5,1.5);

\draw[dotted, gray] (0,2) -- (2,2);
\draw[dotted, gray] (0,1) -- (3,1);

\node at (2,2)[circle,fill,inner sep=1.5pt, blue]{};
\node at (3,1)[circle,fill,inner sep=1.5pt, red]{};

\node at (4.25,-1.2) {\small\textbf{(b) RAP-modulated fluid process}};
\end{scope}
\end{tikzpicture}
\caption{Embedding of a spectrally negative Lévy process into a RAP-modulated fluid process. In (a), $X_t$ drifts upward and jumps downward. In (b), each jump unfolds into a descending sojourn in $\mathcal{Z}^-$ with duration equal to the jump size and level decreasing at unit rate, alternating with ascending drift phases in $\mathcal{Z}^+$. Dotted lines and dots represent the first stage of each process, highlighting that at those levels both processes take the same value.}
\label{fig:embedding_bv}
\end{figure}
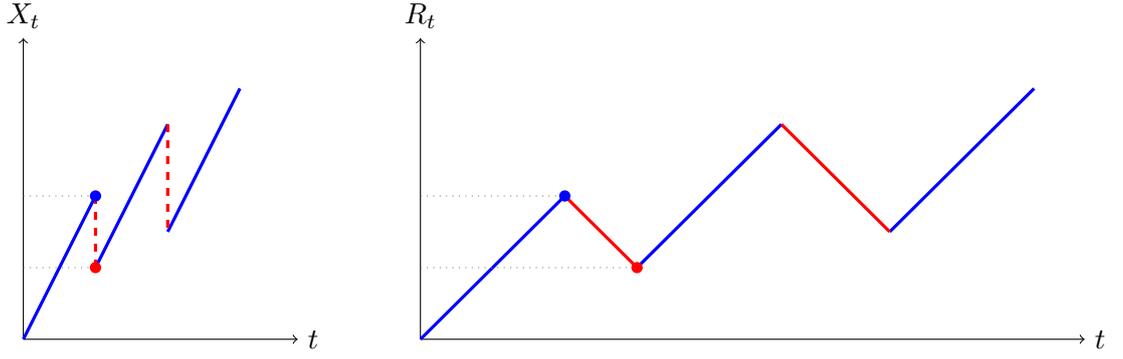
\subsubsection{Scale Function Representation}

To express the hitting probability \eqref{eq:hitting_factor} in terms of the orbit framework, we condition on the orbit value at the first downcrossing of level $-x$. 
Note that in the bounded-variation case with $d > 0$, the downcrossing of the RAP-modulated fluid process occurs during a sojourn in $\mathcal{Z}^-$. That is to say, after a jump from $\mathcal{Z}^+$ to $\mathcal{Z}^-$ and before ending the sojourn of $\mathcal{Z}^-$, the process reaches levels below $-x$. 
Using the previous idea, we obtain the following proposition, which is the matrix-exponential extension of \cite[Theorem 1]{Ivanovs2021} in the bounded variation case.

\begin{proof}[Proof of first part of Theorem \ref{teoScaleFnBV}]
Given that $\bm{O}_x = \bm{a} \in \mathcal{Z}^-$ at the downcrossing instant, the overshoot $Z$ is the residual sojourn time until exit from $\mathcal{Z}^-$.
{That is, after time $\gamma_{\{-x\}}$, the RAP-modulated fluid process drains during the time interval $[\gamma_{\{-x\}},\gamma_{\{-x\}}+Z]$, after which the orbit jumps from $\mathcal{Z}^-$ to $\mathcal{Z}^+$.}
Thus, by \eqref{eq:survival-residual-simple},
\[\PP(Z > z \mid \bm{O}_x = \bm{a}) = \bm{a}e^{\bm{C}^- z}\bm{1}\implies -\frac{\dd}{\dd z}\PP(Z > z \mid \bm{O}_x = \bm{a}) = -\bm{a}e^{\bm{C}^- z}\bm{C}^-\bm{1} = \bm{a}e^{\bm{T} z}\bm{t},\] from which we obtain
\begin{equation*}
\mathbb{E}[\e^{-\Phi_q Z} \mid \bm{O}_x = \bm{a}]
= \int_0^\infty \e^{-\Phi_q z} \bm{a}e^{\bm{T} z}\bm{t} \,\dd z
= \bm{a} (\Phi_q \bm{I} - \bm{T})^{-1}\bm{t}.
\end{equation*}

Define
\begin{equation}\label{eq:nu_bv}
\bm{\nu} := (\Phi_q \bm I - \bm{T})^{-1}\bm{t}.
\end{equation}
Then $\mathbb{E}[\e^{-\Phi_q Z} \mid \bm{O}_x = \bm{a}] = \bm{a}\bm{\nu}$, and by $\{\bm{O}_x = \bm{a}\}=\{\bm{O}_x = \bm{a}, \gamma_{\{-x\}} < \infty\}$ and the tower property,
\begin{equation}\label{eqnBoundedVariationTauxProb}
\begin{split}
\PP(\tau_{\{-x\}} < e_q) 
&= \mathbb{E}\Big[\e^{-\Phi_q Z}\;\mathds{1}\{\gamma_{\{-x\}} < \infty\}\Big]\\
&= \mathbb{E}_{\bm A_0}\Big[\bm{O}_x \bm{\nu}\;\mathds{1}\{\gamma_{\{-x\}} < \infty\}\Big]\\
&= \mathbb{E}_{\bm A_0}\Big[\bm{O}_x\;\mathds{1}\{\gamma_{\{-x\}} < \infty\}\Big]\bm{\nu}.
\end{split}
\end{equation}Starting from $\bm{A}_0 = 1 \in \mathcal{Z}^+$, the downward record formula \eqref{eq:O_x_formula} gives $\mathbb{E}_{\bm A_0}[\bm{O}_x\;\mathds{1}\{\gamma_{\{-x\}} < \infty\}] = \bm{\Psi}\, e^{\bm{G}x}$, where $\bm{G} := \bm{C}^- + \bm{D}^{-+}\bm{\Psi}= \bm{T} + \bm{t}\,\bm{\Psi}$. Therefore,
\begin{equation}\label{eq:hitting_bv_final}
\PP(\tau_{\{-x\}} < e_q) = \bm{\Psi}\, e^{\bm{G}x}\,\bm{\nu}.
\end{equation}
Substituting into  \eqref{eq:Ivanovs_identity} yields \eqref{eq:scale_bv}. 
\end{proof}

In the phase-type case, $\bm{G} = \bm{T} + \bm{t}\,\bm{\Psi}$ has a probabilistic interpretation as a transition rate matrix and $\bm{\Psi}$ represents a minimal nonnegative solution to the Riccati equation. In the matrix-exponential case, these probabilistic properties are lost. The matrix $\bm{G}$ is no longer guaranteed to be a generator, and $\bm{\Psi}$ need not be nonnegative. Nevertheless, both $\bm{\Psi}$ and $\bm{G}$ retain their roles as matrices bridging expectations in the orbit framework: $\bm{\Psi}$ maps initial orbit distributions to expected orbit values at downcrossing, and $\bm{G}$ governs the propagation of these expectations across levels. The algebraic structure ensures that the combination $\bm{\Psi} e^{\bm{G}x}\bm{\nu}$ correctly evaluates to the hitting probability $\PP(\tau_{\{-x\}}<e_q)$, even when the individual components lack direct probabilistic interpretations. The orbit process formalism ensures that the Riccati equation encodes the correct hitting probabilities through the matrix identities \eqref{eq:O_x_formula_minus}--\eqref{eq:O_x_formula}.

\subsubsection{Iterative Algorithm and Limiting Equation}

\begin{proof}[Proof of second part of Theorem \ref{teoScaleFnBV}]
The Riccati equation \eqref{eq:riccati} involves only the matrices $\bm{C}^+$, $\bm{C}^-$, $\bm{D}^{+-}$, and $\bm{D}^{-+}$. The effect of $q$-killing is encoded entirely through termination from $\mathcal{Z}^+$, reflected in the normalization deficit $\bm{C}^+\bm{1} + \bm{D}^{+-}\bm{1} = -q/d < 0$. Since $\bm{\Psi}$ in \eqref{eq:psi_interpretation} is defined via the indicator $\mathds{1}\{\tau_- < \infty\}$, paths that terminate before downcrossing do not contribute to the computation of this matrix.
Substituting the parameters from \eqref{eq:params_bv_plus} into Sylvester's recursion \eqref{eq:psi_sylvester}:
\begin{equation}\label{eq:Sylvester_bv}
-\frac{\lambda + q}{d} \bm{\Psi}_n + \bm{\Psi}_n \bm{T} 
= -\frac{\lambda}{d}\bm{\alpha} - \bm{\Psi}_{n-1}\bm{t}\bm{\Psi}_{n-1},
\qquad n\ge 1,
\end{equation}
where $\bm{\Psi}_0 = \bm{0}$. Since the inverse of $\frac{\lambda + q}{d}\,\bm{I} - \bm{T}$ exists \cite[Theorem 4.1.6]{BladtNielsen2017}, we can isolate $\bm{\Psi}_n$ to obtain \eqref{eq:iteration_bv1}. The convergence $\bm{\Psi} = \lim \bm{\Psi}_n$ follows from \cite[Theorem 4.1]{BeanNguyenNielsenPeralta2022}.
\end{proof}

\begin{remark}\label{rem:bv_iteration_diff}
The iteration \eqref{eq:iteration_bv1} is akin to Equation (3) of \cite{bean2005algorithms}, while the analogous iteration obtained in \cite[Section 4.2]{Ivanovs2021} for phase-type jumps corresponds to Equation (5) of \cite{bean2005algorithms}. In particular, when restricted to the phase-type case, the convergence of both algorithms to the correct matrix follows from \cite{bean2005algorithms}. The reason our iteration differs from that of \cite[Section 4.2]{Ivanovs2021} is that only one algorithmic framework is currently available for RAP-modulated stochastic fluid models \cite{BeanNguyenNielsenPeralta2022}; the richer toolkit available for classical stochastic fluid models admits additional algorithmic variants that are not yet developed in the RAP-modulated setting.
\end{remark}

\subsection{Unbounded Variation Case: $\sigma > 0$}

When $\sigma > 0$, the Lévy process $X$ includes a Brownian component with drift $d$ and variance $\sigma^2$ between jumps. 
We embed such a L\'evy process into a RAP-modulated fluid process, using the setting of Section \ref{subsectionRAPMSFP}, in particular Example~\ref{ex:ME_renewal}. 
Our construction exploits the systematic correspondence between one-sided exit problems for Markov-modulated Brownian motion and those for stochastic fluid processes via Wiener-Hopf factorization of Brownian states, as established by \cite{NguyenPeralta2019}. In the phase-type case, Ivanovs embeds the killed Lévy process into a Markov-modulated Brownian motion (MMBM) and derives scale function representations through MMBM fluctuation theory. By applying the Wiener-Hopf technique of \cite{NguyenPeralta2019}, we can work directly in the RAP-modulated fluid process framework, which naturally accommodates matrix-exponential jump distributions without requiring separate analysis.

To study downcrossings, we decompose the process $X$ into stages, which are the intervals between consecutive jumps of the Poisson process $N$. Each stage is further decomposed into two subintervals as follows. Assume two consecutive jumps of $N$ occur at times $0<\tau_n<\tau_{n+1}$, and the overall infimum of the Brownian path on $[\tau_n,\tau_{n+1}]$ is attained at the unique point $S\in (\tau_n,\tau_{n+1})$. Uniqueness of $S$, and the fact that it is not attained at the endpoints, follow from \cite{MR0433606} (see also \cite[Theorem 2]{MR4130409}). The stage is then decomposed into a first subinterval $[\tau_n,S]$, corresponding to a single \emph{downward} Brownian excursion, and a second subinterval $(S,\tau_{n+1})$, corresponding to an \emph{upward} Brownian excursion. At time $\tau_{n+1}$ the process has a negative jump with matrix-exponential distribution. From the Wiener-Hopf factorization we obtain the joint law of the stage increments $X_{\tau_n}-X_S$ and $X_{\tau_{n+1}-}-X_S$. Equivalently, after recentering the stage at $X_{\tau_n}$, these are the downward and upward Wiener-Hopf factors over the interval $[\tau_n,\tau_{n+1}]$. We then embed such a stage into a stage of a RAP-modulated stochastic fluid process $\{(R_t, \bm{A}_t)\}_t$ with $\pm 1$ rates. The level process starts with a $-1$ slope during a time interval of length $X_{\tau_n}-X_S$, followed by a $+1$ slope during a time interval of length $X_{\tau_{n+1}-}-X_S$, and finally a $-1$ slope of length distributed as matrix-exponential. This matches in distribution the drop from the stage starting level to the stage minimum, the subsequent rise from the minimum to the pre-jump level, and the final jump size. See Figure~\ref{fig:embedding_ubv} for an illustration of the embedding.

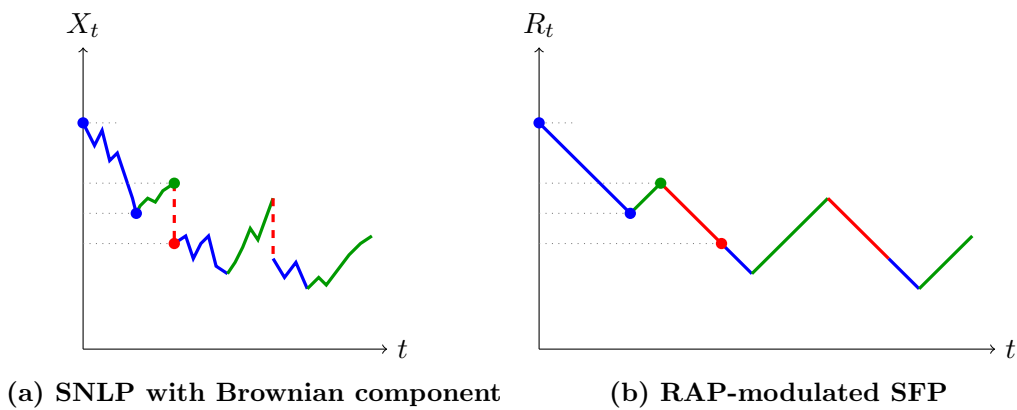
\begin{figure}[ht]
\centering
\begin{tikzpicture}[scale=1.0]

\draw[->] (0,0) -- (4,0) node[right]{$t$};
\draw[->] (0,0) -- (0,4) node[above]{$X_t$};

\def\startlevel{3.0}
\def\firstmin{1.8}
\def\firstend{2.2}
\def\firstjumpsize{0.8}
\def\afterfirstjump{1.4}  
\def\secondmin{1.0}
\def\secondend{2.0}
\def\secondjumpsize{.8}
\def\aftersecondjump{1.2}  

\draw[very thick, blue] 
  (0,\startlevel) -- (0.15,2.7) -- (0.25,2.9) -- (0.35,2.5) -- (0.45,2.6) 
  -- (0.55,2.3) -- (0.65,2.0) -- (0.7,\firstmin) ;
\draw[very thick, green!60!black] 
(0.7,\firstmin) -- (0.75,1.9) 
  -- (0.85,2.0) -- (0.95,1.95) -- (1.05,2.1) -- (1.2,\firstend);

\draw[very thick, red, dashed] (1.2,\firstend) -- (1.2,\afterfirstjump);

\draw[very thick, blue]
  (1.2,\afterfirstjump) -- (1.35,1.5) -- (1.45,1.2) -- (1.55,1.4)
  -- (1.65,1.5) -- (1.75,1.1) -- (1.9,\secondmin);
\draw[very thick, green!60!black] 
(1.9,\secondmin) -- (2.0,1.15)
  -- (2.1,1.35) -- (2.2,1.6) -- (2.3,1.45) -- (2.5,\secondend);

\draw[very thick, red, dashed] (2.5,\secondend) -- (2.5,\aftersecondjump);

\draw[very thick, blue]
  (2.5,\aftersecondjump) -- (2.65,0.95) -- (2.8,1.15) -- (2.95,0.8);
\draw[very thick, green!60!black] 
  (2.95,0.8)-- (3.1,0.95) -- (3.2,0.85) -- (3.35,1.05) -- (3.5,1.25) 
  -- (3.65,1.4) -- (3.8,1.5);

\draw[dotted, gray] (0,\startlevel) -- (0.5,\startlevel);
\node at (0,\startlevel)[circle,fill,inner sep=1.5pt, blue]{};
\draw[dotted, gray] (0,\firstmin) -- (0.7,\firstmin);
\node at (0.7,\firstmin)[circle,fill,inner sep=1.5pt, blue]{};
\draw[dotted, gray] (0,\firstend) -- (1.2,\firstend);
\node at (1.2,\firstend)[circle,fill,inner sep=1.5pt, green!60!black]{};
 \draw[dotted, gray] (0,\afterfirstjump) -- (1.2,\afterfirstjump);
\node at (1.2,\afterfirstjump)[circle,fill,inner sep=1.5pt, red]{};

\node at (2.25,-0.6) {\small\textbf{(a) SNLP with Brownian component}};

\begin{scope}[xshift=6cm]
\draw[->] (0,0) -- (6,0) node[right]{$t$};
\draw[->] (0,0) -- (0,4) node[above]{$R_t$};

\draw[very thick, blue] (0,\startlevel) -- (1.2,\firstmin);
\draw[very thick, green!60!black] (1.2,\firstmin) -- (1.6,\firstend);
\draw[very thick, red] (1.6,\firstend) -- (2.4,\afterfirstjump);

\draw[very thick, blue] (2.4,\afterfirstjump) -- (2.8,\secondmin);
\draw[very thick, green!60!black] (2.8,\secondmin) -- (3.8,\secondend);
\draw[very thick, red] (3.8,\secondend) -- (4.6,\aftersecondjump);

\draw[very thick, blue] (4.6,\aftersecondjump) -- (5,0.8);
\draw[very thick, green!60!black] (5,0.8) -- (5.7,1.5);

\draw[dotted, gray] (0,\startlevel) -- (0.5,\startlevel);
\draw[dotted, gray] (0,\firstend) -- (1.6,\firstend);
\draw[dotted, gray] (0,\firstmin) -- (1.2,\firstmin);
\node at (1.2,\firstmin)[circle,fill,inner sep=1.5pt, blue]{};
\draw[dotted, gray] (0,\afterfirstjump) -- (2.4,\afterfirstjump);
\node at (0,\startlevel)[circle,fill,inner sep=1.5pt, blue]{};
\node at (1.6,\firstend)[circle,fill,inner sep=1.5pt, green!60!black]{};
\node at (2.4,\afterfirstjump)[circle,fill,inner sep=1.5pt, red]{};

\node at (3.15,-0.6) {\small\textbf{(b) RAP-modulated SFP}};
\end{scope}

\end{tikzpicture}
\caption{Embedding of a spectrally negative Lévy process with Brownian component into a RAP-modulated fluid process. In (a), we show the path from the initial point up to hitting the infimum, the path from hitting the infimum up to just before the downward jump, and the downward jump. In (b), we highlight the first stage of the RAP-modulated fluid process: starting with a downward phase ($\mathcal{Z}_1^-$, orange), then an upward phase ($\mathcal{Z}_1^+$, green), while jumps unfold into descending sojourns ($\mathcal{Z}_2^-$, red). Dotted lines and dots  represent the first stage of each process, highlighting that at those levels both processes take the same value.
}
\label{fig:embedding_ubv}
\end{figure}

\subsubsection{Wiener-Hopf Factorization and Embedding}

The following result provides us with the law of the Brownian motion at the appropriate times discussed above. 
Its proof can be found in \cite{BladtNielsen2017}. 

\begin{theorem}[Wiener-Hopf factorisation for Brownian motion]
Let $\{W_t\}_{t\geq 0}$ be a Brownian motion with variance $\sigma^2>0$, drift $d\in \mathbb{R}$ and starting point $W_0=0$. 
Let $U$ be a stopping time, and for fixed $\lambda+q > 0$ let $E\sim {\rm Exp}(\lambda+q)$, independent of $\{W_t\}_{t\geq 0}$. 
Then, $W_U-\min_{t\in [0,E]}W_{U+t}$ and $W_{U+E}-\min_{t\in [0,E]}W_{U+t}$ are independent and exponentially distributed random variables with respective rates
\begin{equation}\label{eq:wiener_hopf_factors}
\omega_{\lambda+q} = \frac{\sqrt{d^2 + 2\sigma^2 (\lambda+q)} + d}{\sigma^2}, \qquad \eta_{\lambda+q} = \frac{\sqrt{d^2 + 2\sigma^2 (\lambda+q)} - d}{\sigma^2}.
\end{equation}
\end{theorem}

Employing the Wiener-Hopf factorization, we embed the killed process $X$ into a RAP-modulated stochastic fluid process with state space $\mathcal{Z} = \mathcal{Z}^- \cup \mathcal{Z}^+$, where:
\begin{itemize}
\item $\mathcal{Z}^- = \mathcal{Z}_1^- \cup \mathcal{Z}_2^-$ decomposes into:
\begin{itemize}
\item[1)] $\mathcal{Z}_1^-$: downward Brownian excursion phase (level decreases at rate $-1$)
\item[2)] $\mathcal{Z}_2^-$: ME jump drain phase (level decreases at rate $-1$)
\end{itemize}
\item $\mathcal{Z}^+ = \mathcal{Z}_1^+$: upward Brownian excursion phase (level increases at rate $+1$)
\end{itemize}

The RAP-modulated fluid process operates in stages as follows (see Figure~\ref{fig:embedding_ubv}):\footnote{I added a ton of detail. Feel free to cut wherever seems necessary.}
\begin{enumerate}
\item \textbf{Downward Brownian excursion}: Starting from $\mathcal{Z}_1^-$, the level decreases at rate $-1$ for a duration $\mathrm{Exp}(\omega_{\lambda+q})$. 
This segment should be interpreted as the Wiener-Hopf \emph{downward factor} over the exponential horizon $E\sim\mathrm{Exp}(\lambda+q)$, namely the total drop from the current level to the running minimum attained before the Brownian stage ends. 
\item \textbf{Upward Brownian excursion}: The orbit transitions to $\mathcal{Z}_1^+$, and the level then increases at rate $+1$ for a duration $\mathrm{Exp}(\eta_{\lambda+q})$.
Likewise, this segment represents the Wiener-Hopf \emph{upward factor}, that is, the amount by which the terminal Brownian position at the end of the stage lies above that running minimum. 
\item \textbf{Jump or termination}: At the end of the upward excursion:
\begin{itemize}
\item With probability $\lambda/(\lambda + q)$, a negative jump occurs: the orbit enters $\mathcal{Z}_2^-$ with initial distribution $\bm{\alpha}$, and the level decreases at rate $-1$ for a duration $\mathrm{ME}(\bm{\alpha}, \bm{T})$, unfolding the jump size. Upon exit from $\mathcal{Z}_2^-$, the orbit returns to $\mathcal{Z}_1^-$ and a new stage begins.
\item With probability $q/(\lambda + q)$, the process terminates (representing the $q$-killing).
\end{itemize}
\end{enumerate}

We now explain how those probabilities of either jump or termination, map precisely the $q$-killing of the L\'evy process.
The key point is that the RAP encodes the termination at the `end' of a Brownian stage, not the corresponding time at which killing occurs in the L\'evy process. 
Since only the law of the completed stage matters for the matrix-analytic representation, we are free to place that Bernoulli decision at the exit from $\mathcal{Z}_1^+$: conditionally on the Brownian stage ending at an $\mathrm{Exp}(\lambda+q)$-distributed time, the event that ended the stage is a jump with probability $\lambda/(\lambda+q)$ and a killing with probability $q/(\lambda+q)$.
By the Wiener-Hopf factorization, that one-stage law is completely described by: a downward factor with rate $\omega_{\lambda+q}$,  an upward factor with rate $\eta_{\lambda+q}$, together with a Bernoulli mark telling us whether the stage ended because either the jump clock or the killing clock rang first. 
The latter and the properties of the geometric distribution, justify the given probabilities of termination.

The embedding can be viewed as an instance of Example~\ref{ex:ME_renewal} with $n_+=1$, $n_-=2$, and the following parameters. 
By the above discussion, the transition probabilities are
\[
p^{-+}_{11} = 1, \qquad p^{+-}_{12} = \frac{\lambda}{\lambda+q}, \qquad p^{--}_{21} = 1,
\]
with all other transition probabilities zero. The first equality reflects that every downward excursion is necessarily followed by an upward one (no termination during $\mathcal{Z}_1^-$); the second encodes the $q$-killing during the upward phase; and the third reflects that every matrix-exponential drain phase is preceded by a complete upward excursion (no termination during $\mathcal{Z}_2^-$). The initial orbit values are $\bm{\alpha}_1^+ = 1$, $\bm{\alpha}_1^- = 1$, and $\bm{\alpha}_2^- = \bm{\alpha}$.

The orbit parameters are, in block form with rows and columns ordered as $(\mathcal{Z}_1^-, \mathcal{Z}_2^-)$ within $\mathcal{Z}^-$:
\begin{equation}\label{eq:params_ubv}
\bm{C}^- = \begin{pmatrix} -\omega_{\lambda+q} & \bm{0} \\ \bm{t} & \bm{T} \end{pmatrix}, \qquad \bm{C}^+ = -\eta_{\lambda+q},
\end{equation}
\begin{equation}\label{eq:params_ubv1}
\bm{D}^{-+} = \begin{pmatrix} \omega_{\lambda+q} \\ \bm{0} \end{pmatrix}, \qquad \bm{D}^{+-} = \begin{pmatrix} 0 & \eta_{\lambda+q}\dfrac{\lambda}{\lambda+q}\,\bm{\alpha} \end{pmatrix},
\end{equation}
where $\bm{t} := -\bm{T}\bm{1}$.

\begin{remark}
When $q>0$, the normalization condition for $\mathcal{Z}^+$ holds with strict inequality: $C^+\bm{1} + \bm{D}^{+-}\bm{1} = -\eta_{\lambda+q}q/(\lambda+q) < 0$, with the deficit representing termination at rate $\eta_{\lambda+q}q/(\lambda+q)$ from $\mathcal{Z}^+$. When $q=0$ there is no killing, the normalization for $\mathcal{Z}^+$ holds with equality, and the orbit process is non-terminating. In both cases, for $\mathcal{Z}_1^-$ the normalization holds with equality: $-\omega_{\lambda+q} + \omega_{\lambda+q} = 0$, and for $\mathcal{Z}_2^-$ it also holds with equality: $\bm{T}\bm{1} + \bm{t} = \bm{0}$. As explained earlier, termination is absorbed entirely into $\mathcal{Z}^+$: the $q$-killing acts on the combined Brownian stage rather than on either excursion individually, so no termination occurs from $\mathcal{Z}_1^-$ or $\mathcal{Z}_2^-$. This implements the $q$-killing via the termination mechanism of Section~\ref{subsectionRAPMSFP}, avoiding the need for an explicit cemetery state.
\end{remark}

\subsubsection{Scale Function Representation}

\begin{proof}[Proof of first part of Theorem \ref{teoScaleFnUV}]
Let $\bm{e}_1^-$ denote the unit column vector of dimension $p+1$ with one in the first entry, so that $(\bm{e}_1^-)^\intercal$ selects the initial Brownian downward state. Under the block ordering $(\mathcal{Z}_1^-,\mathcal{Z}_2^-)$, the region $\mathcal{Z}_1^-$ is represented by the singleton state $\{(\bm{e}_1^-)^\intercal\}$, whereas $\mathcal{Z}_2^-$ consists of a subset of all row vectors of the form $(0,\bm{a})$ with $\bm{a}\in\mathbb{R}^{1\times p}$ and $\bm{a}\bm{1}=1$. Thus the first coordinate corresponds to the Brownian downward phase, and the remaining $p$ coordinates correspond to the matrix-exponential drain phase.
Note that $\{\tau^-_{\{-x\}} < e_q\} = \{\gamma_{\{-x\}} < \infty\}$ in the RAP-modulated embedding, since the $q$-killing is encoded as termination from $\mathcal{Z}^+$: the orbit reaches level $-x$ if and only if the stage when it terminates happens after reaching $-x$, and the latter happens if and only if the excursion where the L\'evy process is killed occurs after hitting $(-\infty,-x)$. 
This equality is realized under the stagewise coupling induced by the Wiener-Hopf decomposition. 

We use \eqref{eq:hitting_factor} and decompose on the events $\{Z=0\}$ and $\{Z>0\}$, obtaining
\[
\PP(\tau_{\{-x\}} < e_q) = \mathbb{P}\big(\tau^-_{\{-x\}} < e_q,Z=0\big)+\mathbb{E}\Big[\e^{-\Phi_q Z}\;\mathds{1}\big\{\tau^-_{\{-x\}} < e_q,Z>0\big\}\Big].
\]
When $Z=0$ the process hits $-x$ by a downward Brownian excursion, so $\mathds{1}\{Z=0\}=\bm O_{x}\bm{e}_1^-$, and thus
\[
\mathbb{P}\big(\tau^-_{\{-x\}} < e_q,Z=0\big)=\mathbb{E}_{(\bm{e}_1^-)^\intercal}\big[\bm O_x\mathds{1}\{\gamma_{\{-x\}}<\infty\}\big]\bm{e}_1^-.
\]
From \eqref{eq:O_x_formula_minus}, we obtain
\[
\mathbb{P}\big(\tau^-_{\{-x\}} < e_q,Z=0\big)=(\bm{e}_1^-)^\intercal e^{\bm{G}x}\bm{e}_1^-,
\]
where $\bm{G}=\bm{C}^- + \bm{D}^{-+}\bm{\Psi}$.

On the other hand, $\{\tau^-_{\{-x\}} < e_q, Z>0\} = \{\gamma_{\{-x\}}<\infty, \bm{O}_x\in \mathcal{Z}^-_2\}$. Given $\bm{O}_x = (0,\bm{a}) \in \mathcal{Z}_2^-$ with $\bm{a} \in \mathbb{R}^{1\times p}$, the overshoot $Z$ is the residual sojourn time in $\mathcal{Z}_2^-$. Since $\bm{C}_{22}^- = \bm{T}$, Equation \eqref{eq:residual_sojourn_component} gives the survival function $\bm{a}e^{\bm{T}z}\bm{1}$. By the same computation as in the bounded variation proof \eqref{eqnBoundedVariationTauxProb},
\[
\mathbb{E}_{(\bm{e}_1^-)^\intercal}\big[\e^{-\Phi_q Z} \mid \bm{O}_x = (0,\bm{a})\big] = \bm{a}(\Phi_q\bm{I}-\bm{T})^{-1}\bm{t} = \bm{a}\,\bm{\nu}.
\]
To express this in terms of the full $(p+1)$-dimensional orbit vector $\bm{O}_x = (0,\bm{a})$, we use the identity matrix and the block extraction
\[
\bm{a} = (0,\bm{a})\begin{pmatrix}\bm{0}\\\bm{I}_p\end{pmatrix},
\]
so that
\[
\mathbb{E}_{(\bm{e}_1^-)^\intercal}\big[\e^{-\Phi_q Z} \mid \bm{O}_x=(0,\bm{a})\big] = (0,\bm{a})\begin{pmatrix}\bm{0}\\\bm{I}_p\end{pmatrix}\bm{\nu} = (0,\bm{a})\begin{pmatrix}0\\\bm{\nu}\end{pmatrix}.
\]
Thus,
\[
\mathbb{E}\Big[\e^{-\Phi_q Z}\;\mathds{1}\big\{\tau^-_{\{-x\}} < e_q,Z>0\big\}\Big] = \mathbb{E}_{(\bm{e}_1^-)^\intercal}\Big[\bm{O}_x\begin{pmatrix}0\\\bm{\nu}\end{pmatrix}\mathds{1}\{\gamma_{\{-x\}} < \infty, \bm{O}_x\in\mathcal{Z}_2^-\}\Big].
\]
Since on the event $\{\bm{O}_x\in\mathcal{Z}_2^-\}$ we have $\bm{O}_x=(0,\bm{a})$ with $\bm{a}\bm{1}=1$, and on the complementary event $\{\bm{O}_x\in\mathcal{Z}_1^-\}$ we have $\bm{O}_x\begin{pmatrix}0\\\bm{\nu}\end{pmatrix}=0$, the indicator $\mathds{1}\{\bm{O}_x\in\mathcal{Z}_2^-\}$ can be dropped and we may write
\[
\mathbb{E}\Big[\e^{-\Phi_q Z}\;\mathds{1}\big\{\tau^-_{\{-x\}} < e_q,Z>0\big\}\Big] = \mathbb{E}_{(\bm{e}_1^-)^\intercal}\Big[\bm{O}_x\,\mathds{1}\{\gamma_{\{-x\}} < \infty\}\Big]\begin{pmatrix}0\\\bm{\nu}\end{pmatrix} = (\bm{e}_1^-)^\intercal e^{\bm{G}x}\begin{pmatrix}0\\\bm{\nu}\end{pmatrix}.
\]
Combining both cases and setting $\bm{V}:=\begin{pmatrix}1\\\bm{\nu}\end{pmatrix}\in\mathbb{R}^{p+1}$, we obtain
\begin{equation}\label{eq:hitting_ubv}
\PP(\tau_{\{-x\}} < e_q) = (\bm{e}_1^-)^\intercal e^{\bm{G}x}\bm{V}, \qquad x \ge 0.
\end{equation}
The latter and \eqref{eq:Ivanovs_identity} imply \eqref{eq:scale_ubv}.
\end{proof}

\subsubsection{Iterative Algorithm and Limiting Equations}

\begin{proof}[Proof of second part of Theorem \ref{teoScaleFnUV}]
Substituting the parameters \eqref{eq:params_ubv} into the Sylvester recursion \eqref{eq:psi_sylvester}, write $\bm{\Psi}_n=(\psi_n,\bm{\pi}_n)$ with $\bm{\pi}_n\in\mathbb{R}^{1\times p}$. The two blocks yield the coupled system
\begin{equation}\label{eq:sylvester_ubv_scalar1}
\omega_{\lambda+q}\psi^2_{n-1}-(\omega_{\lambda+q}+\eta_{\lambda+q})\psi_n+\bm \pi_n\bm t=0,
\end{equation}
\begin{equation}\label{eq:sylvester_ubv_vector1}
\bm{\pi}_n\,(\eta_{\lambda+q} \bm{I}-\bm{T}) = \eta_{\lambda+q}\frac{\lambda}{\lambda+q}\,\bm{\alpha} + \psi_{n-1}\omega_{\lambda+q}\,\bm{\pi}_{n-1}.
\end{equation}
Set
\begin{equation}\label{eq:ab_def1}
a_n:=\omega_{\lambda+q}(1-\psi_n),\qquad \bm{b}_n:=\omega_{\lambda+q}\bm{\pi}_n.
\end{equation}

By isolating $\bm{\pi}_n$ and $\psi_n$, then $\bm{b}_n$ can be updated sequentially from the previous iterate $(a_{n-1},\bm{b}_{n-1})$, whereas $a_n$ depends on $(a_{n-1},\bm{b}_{n})$. The latter yields the true Sylvester recursions \eqref{eq:b_recursion} and \eqref{eq:a_recursion} after substituting the Wiener-Hopf identities
\begin{equation}\label{eq:wh_identities}
\omega_{\lambda+q} - \eta_{\lambda+q} = \frac{2d}{\sigma^2}, \quad
\omega_{\lambda+q}\eta_{\lambda+q} = \frac{2(\lambda+q)}{\sigma^2}, \quad \omega_{\lambda+q} + \eta_{\lambda+q} = \frac{2\sqrt{d^2+2\sigma^2(\lambda+q)}}{\sigma^2}.
\end{equation}
Since $\bm{\Psi}_n \to \bm{\Psi}$ entrywise by \cite[Corollary 4.2]{BeanNguyenNielsenPeralta2022}, the pair $(a_n, \bm{b}_n)$ converges to some $(a, \bm{b})$. Monotone convergence cannot be expected in general, as the entries of $\bm{\Psi}_n$ need not be nonnegative or ordered.

Computing $\bm{G} = \bm{C}^- + \bm{D}^{-+}\bm{\Psi}$ with the block structure from \eqref{eq:params_ubv}, we obtain:
\begin{equation}\label{eq:G_ubv}
\bm{G} = \begin{pmatrix}-\omega_{\lambda+q} & \bm{0} \\ \bm{t} & \bm{T}\end{pmatrix} + \begin{pmatrix}\omega_{\lambda+q} \\ \bm{0}\end{pmatrix}\begin{pmatrix}\psi & \bm{\pi}\end{pmatrix} = \begin{pmatrix}-a & \bm{b} \\ \bm{t} & \bm{T}\end{pmatrix}.
\end{equation}
Partitioning $\bm{\Psi} = \begin{pmatrix}\psi & \bm{\pi}\end{pmatrix}$ with $\psi \in \mathbb{R}$ and $\bm{\pi} \in \mathbb{R}^{1 \times p}$, and defining
\begin{equation}\label{eq:ab_def}
a := \omega_{\lambda+q}(1-\psi), \qquad \bm{b} := \omega_{\lambda+q}\bm{\pi},
\end{equation}
substituting \eqref{eq:params_ubv}-\eqref{eq:params_ubv1} into \eqref{eq:riccati}, 
and using the limit of Equation \eqref{eq:sylvester_ubv_scalar1},  the scalar component becomes
\[
\begin{split}
\bm b\bm t& =\bm \pi \bm t\omega_{\lambda+q}\\
&= (-\omega_{\lambda+q}\psi+\omega_{\lambda+q}+\eta_{\lambda+q})\omega_{\lambda+q}\psi\\
&= (\omega_{\lambda+q}(1-\psi)+\eta_{\lambda+q})\omega_{\lambda+q}\psi\\
& = (a+\eta_{\lambda+q})(\omega_{\lambda+q}-a).
\end{split}
\]
Using \eqref{eq:wh_identities} and multiplying by $-\sigma^2$, the latter becomes the quadratic equation \eqref{eq:ivanovs_scalar}. Equation \eqref{eq:ivanovs_vector} is obtained from the limiting form of \eqref{eq:sylvester_ubv_vector1} by using again the first identity in \eqref{eq:wh_identities}.

%
%
The limit determines $\bm{\Psi}$ and $\bm{G}$ via \eqref{eq:Psi_from_ab} and \eqref{eq:G_ubv}, completing the computation of the scale function \eqref{eq:scale_ubv}.
\end{proof}

As in the bounded variation case, the matrix $\bm{G}$ loses its probabilistic interpretation when the jump distribution is matrix-exponential but not phase-type: $a$ need not be positive and $\bm{b}$ need not have nonnegative entries. Nevertheless, $\bm{\Psi}$ and $\bm{G}$ retain their algebraic roles in the orbit framework, and the combination $(\bm{e}_1^-)^\intercal e^{\bm{G}x}\bm{V}$ correctly evaluates to $\PP(\tau_{\{-x\}}<e_q)$. The extension from phase-type to matrix-exponential preserves the fundamental algebraic structure: the orbit process framework ensures that the same quadratic form arises, with $(\bm{\alpha}, \bm{T})$ entering in the same positions as in the phase-type case.

\section{Conclusion}\label{sec:conclusion}

We have extended the scale function representation of \cite{Ivanovs2021} from spectrally negative Lévy processes with phase-type jumps to those with general matrix-exponential jumps. The extension is non-trivial: the probabilistic arguments employed by Ivanovs rely fundamentally on the Markov-modulated Brownian motion framework, which requires the jump distribution to admit a phase-type representation with an underlying continuous-time Markov chain. Matrix-exponential distributions that are not phase-type lack this Markovian structure, and several convenient probabilistic interpretations are lost during the algebraic ME-ification.

The key to overcoming this obstacle is the orbit process framework introduced by \cite{BeanNguyenNielsenPeralta2022}. By embedding the Lévy process into a RAP-modulated stochastic fluid process driven by an orbit process rather than a Markov chain, we accommodate the algebraic generality of matrix-exponential distributions. A crucial technical insight is the systematic correspondence between one-sided exit problems for Markov-modulated Brownian motion and those for stochastic fluid processes via Wiener-Hopf factorization of Brownian states, established by \cite{NguyenPeralta2019}. This technique allows results derived in the Brownian setting with rational jumps to be translated directly to the RAP-modulated fluid process context, avoiding the need for separate analysis and enabling our matrix-exponential extension. The Riccati equations governing first passage probabilities in RAP-modulated fluid processes reduce to the same quadratic matrix equations derived by Ivanovs, without requiring limiting approximations or probabilistic interpretations.

Our construction accommodates terminating orbit processes through relaxed normalization conditions, implementing killing mechanisms via substochastic transitions rather than explicit cemetery states. While \cite{BeanNguyenNielsenPeralta2022} emphasized non-terminating RAP-modulated fluid processes, we use their analytical framework to extend naturally to the terminating case, as the relevant expectations condition on downcrossing events that precede termination. This flexibility is essential for modeling killed Lévy processes and enriches the applicability of the orbit framework.

Our work contributes to the broader effort of developing the analytically challenging matrix-exponential toolkit. Matrix-exponential distributions provide a dense class of distributions on the positive half-line with tractable transforms, but their non-Markovian nature presents obstacles for many standard techniques. Orbit processes and RAP-modulated stochastic fluid processes offer a systematic framework for handling matrix-exponential structures, extending results from the phase-type literature while preserving computational feasibility. The iterative algorithms we derive retain explicit matrix-analytic recursions with convergence guaranteed by the RAP-modulated fluid framework, and in the phase-type case they recover the same fixed-point equations as Ivanovs's approach. This demonstrates that the algebraic structure underlying these methods is robust to the matrix-exponential generalization.

We demonstrate here the utility of orbit representations. While orbit processes sacrifice some of the intuitive probabilistic content of Markov chains, they retain the essential algebraic properties needed for fluctuation theory and other analytical applications. As the matrix-exponential toolkit continues to develop, orbit processes are likely to play an increasingly central role, providing the natural setting for generalizations beyond the phase-type regime. 

Future work may extend these techniques to related Lévy process models. Since matrix-exponential distributions are precisely equivalent to distributions with rational Laplace transforms, our results exhaust the class of spectrally negative Lévy processes admitting rational characteristics. Natural directions include extension to two-sided Lévy processes with matrix-exponential jumps in both directions, though such generalizations would likely not yield the same clean matrix-exponential representations for scale functions. The orbit framework can also accommodate infinite-dimensional phase-type distributions, which are relevant for heavy-tailed models. The terminating orbit framework developed here may prove useful in other contexts requiring absorption or killing mechanisms. These extensions will require additional technical development but follow the same conceptual approach: replacing Markovian structures with orbit dynamics facilitates the analysis while also preserves the underlying algebraic relationships.

\bibliography{references}
\bibliographystyle{alpha}

\end{document}